\documentclass{siamart}

\usepackage{amsfonts}
\usepackage{graphicx}
\usepackage{epstopdf}
\usepackage{algorithmic}
\usepackage{amsopn}
\usepackage{cite}
\usepackage{amsmath}
\usepackage[toc, page]{appendix}
\usepackage{enumitem}
\usepackage{amssymb}
\usepackage{indentfirst}
\usepackage{caption}
\usepackage{subcaption}
\usepackage{tikz}
\usepackage[group-separator={,},group-minimum-digits=4]{siunitx}
\usetikzlibrary{calc, arrows,shapes}

\DeclareMathOperator*{\argmin}{arg\,min}
\DeclareMathOperator*{\argmax}{arg\,max}
\DeclareMathOperator{\tr}{tr}
\DeclareMathOperator{\cSpan}{Span}

\DeclareMathOperator{\eig}{eig}
\DeclareMathOperator{\dist}{dist}

\def\vx{\boldsymbol{x}}
\def\vy{\boldsymbol{y}}

\def\mA{\boldsymbol{A}}

\def\mC{\boldsymbol{C}}
\def\mH{\boldsymbol{H}}
\def\mI{\boldsymbol{I}}
\def\mL{\boldsymbol{L}}
\def\mP{\boldsymbol{P}}
\def\mQ{\boldsymbol{Q}}
\def\mR{\boldsymbol{R}}
\def\mU{\boldsymbol{U}}

\def\mX{\boldsymbol{X}}
\def\mZ{\boldsymbol{Z}}

\def\v0{\boldsymbol{0}}

\def\blamb{\boldsymbol{\Lambda}}

\def\blam{\boldsymbol{\lambda}}

\def\bbH{\mathbb{H}}
\def\bbN{\mathbb{N}}
\def\bbR{\mathbb{R}}

\def\cS{\mathcal{S}^D}
\def\cSp{\mathcal{S}_{+}^D}
\def\cSpp{\mathcal{S}_{++}^D}

\def\cE{\mathcal{E}}

\def\cN{\mathcal{N}}
\def\cX{\mathcal{X}}
\def\cY{\mathcal{Y}}

\ifpdf
  \DeclareGraphicsExtensions{.eps,.pdf,.png,.jpg}
\else
  \DeclareGraphicsExtensions{.eps}
\fi

\newcommand{\TheTitle}{Distributed Robust Subspace Recovery}
\newcommand{\TheAuthors}{Vahan Huroyan and Gilad Lerman}
\newcommand*{\QED}{\hfill\ensuremath{\square}}%

\headers{Distributed Robust Subspace Recovery}{\TheAuthors}

\title{{\TheTitle}}

\author{
  Vahan Huroyan\thanks{School of Mathematics, University of Minnesota, Twin Cities $\{$\email{huroy002, lerman}$\}$\email{@umn.edu}}
  \and
  Gilad Lerman\footnotemark[1]
}

\ifpdf
\hypersetup{
  pdftitle={\TheTitle},
  pdfauthor={\TheAuthors}
}
\fi

\begin{document}

\newtheorem{mydef}{Definition}
\newtheorem{Theorem}{Theorem}

\maketitle

% REQUIRED
\begin{abstract}
%We study Robust Subspace Recovery (RSR) in distributed settings. We consider a huge dataset in an ad hoc network without a central processor, where each node has access only to one chunk of the dataset. We assume that part of the whole dataset lies around a low-dimensional subspace and the other part is composed of outliers that lie away from that subspace. The goal is to recover the underlying subspace for the whole dataset, without transferring the data itself between the nodes. We apply the Consensus Based Gradient method for the Geometric Median Subspace algorithm for RSR. We propose an iterative solution for the local dual minimization problem and establish its $r$-linear convergence. %We show that this mathematical framework also extends to two simpler problems:  Principal Component Analysis  and the geometric median.
%We also explain how to distributedly implement the Reaper and Fast Median Subspace algorithms for RSR.
%We demonstrate the competitive performance of our algorithms for both synthetic and real data.

We propose distributed solutions to the problem of Robust Subspace Recovery (RSR). Our setting assumes a huge dataset in an ad hoc network without a central processor, where each node has access only to one chunk of the dataset. Furthermore, part of the whole dataset lies around a low-dimensional subspace and the other part is composed of outliers that lie away from that subspace. The goal is to recover the underlying subspace for the whole dataset, without transferring the data itself between the nodes. We first apply the Consensus Based Gradient method to the Geometric Median Subspace algorithm for RSR. For this purpose, we propose an iterative solution for the local dual minimization
problem and establish its r-linear convergence.
We then explain how to distributedly implement the Reaper and Fast Median Subspace algorithms for RSR.
The proposed algorithms display competitive performance on both synthetic and real data.
\end{abstract}

% REQUIRED

% REQUIRED

\section{Introduction}
Distributed computing is a central theme in modern computation. Its setting includes a system with multiple components, which communicate and coordinate in order to achieve their common computational goal. A special distributed setting assumes a central processor, which is connected to all other processors. This processor contains no data, but has enough memory to handle some computations, such as averaging communicated estimates. A more general distributed setting assumes an arbitrarily connected network of processors, among which the data is partitioned. Each processor computes a local estimate of the desired output based on its local data and on estimates passed by its neighbors. Then, it communicates its estimate to its neighbors. This procedure iterates until convergence.

Some common approaches for solving distributed computing problems are the diffusion method \cite{chen2012diffusion}, the Consensus-Based Gradient Ascent (CBGA) \cite{bertrand2011consensus,johansson2009on,boyd2011distributed, rabbat2005generalized}, the distributed subgradient method \cite{Nedic2009distributed, nedic2010distributed} and the Consensus Alternating Direction Method of Multipliers (CADMM) \cite{boyd2011distributed,nedic2010distributed,shi2014on, chang2015multi, mateos2010distributed}. Some of these algorithms have been successfully adapted to important applied problems of signal processing and wireless communications \cite{schizas2008consensus,mateos2010distributed,zhu2010distributed,forero2010consensus}. Various distributed algorithms have been proposed for the important problem of Principal Component Analysis (PCA) and related problems, such as the total least squares. Most of them are for centrally-processed networks \cite{qu2002principal,qi2004global,bai2005principal,liang2014improved,meng2012distributed,valcarcel2010consensus}, but some of them are for arbitrarily connected networks \cite{aduroja2013distributed,bertrand2011consensus}.  To the best of our knowledge there are no distributed algorithms for robust versions of PCA.

This work discusses distributed algorithms for Robust Subspace Recovery (RSR) with arbitrarily connected networks. RSR is an alternative paradigm for PCA that is more robust to outliers.
%(since PCA is based on projected-variance maximization, it is not robust to outliers).
The underlying problem of RSR assumes data points, composed of inliers and outliers, where the inliers are well-explained by an affine low-dimensional subspace and the outliers come from a different model. The goal is to recover the underlying subspace in the presence of outliers. A careful review of the problem and its solutions appears in \cite{RSR_review_LM18}.

We first suggest a distributed implementation for the Geometric Median Subspace (GMS) \cite{gms} algorithm for RSR, which applies to arbitrarily connected networks. We propose an iterative algorithm for the local dual problem and establish its $r$-linear convergence (defined later in \Cref{def:rlinear}).
%Our mathematical framework for distributed GMS extends to two simpler problems: distributed PCA and distributed geometric median.
We also propose distributed implementations for two other RSR algorithms: Reaper \cite{reaper} and FMS \cite{maunu2014fast}. This is done by iterative application of distributed PCA. On the other hand, the GMS implementation does not iterate the distributed scheme and is thus more efficient %than distributed FMS and distributed Reaper
in terms of the communication cost. We remark that the theorems for robustness of GMS, Reaper and FMS carry over to our distributed setting.

The paper is organized as follows: \S\ref{sec:cbga} contains a short introduction to CBGA and its convergence analysis; \S\ref{sec:distgms} proposes the distributed CBGA algorithm for GMS and discusses its various properties; \S\ref{sec:distreaperfms} proposes immediate distributed implementations for the Reaper and FMS algorithms; and \S\ref{sec:numerical} concludes with numerical experiments that test the proposed algorithms for distributed RSR. Appendices~\ref{sec:distpca} and~\ref{sec:geommedian} use ideas of \S\ref{sec:cbga} to solve the problems of distributed PCA and distributed geometric median.
\Cref{sec:cadmm} explains how to apply CADMM instead of CBGA for a distributed version of GMS. \Cref{append:proofs} provides details of proofs of all theoretical statements.

%----------------------------------------------------------------------------------------

\section{Review of Consensus-Based Gradient Ascent (CBGA)}
\label{sec:cbga}
The setting of CBGA \cite{rabbat2005generalized} assumes a connected network, with $K$ nodes and $M$ edges. It also assumes a convex set of matrices $S \subseteq \bbR^{D \times D}$  and convex functions $F_1, \dots, F_K$ on $S,$ associated with the $K$ nodes.  The goal is to minimize $\sum_{k = 1}^K {F_k}$ over $S$, where each node $k$ has only access to $F_k$ and may communicate to its neighbors. The consensus-based formulation of this problem uses local neighborhoods as follows. For  $1 \le k \le K,$ let $\cN_k$ denote the set of all nodes connected (by an edge) to the node $k.$ The desired problem, $\min_{\mQ \in S} \sum_{k = 1}^K F_k(\mQ),$ can be computed locally as follows:
\begin{equation}
\min \limits_{\mQ_1, \dots \mQ_K \in S} \sum \limits_{k = 1}^K F_k(\mQ_k), \text{where } \mQ_k = \mQ_q, \forall 1 \le k \le K,  q \in \cN_k, q < k.
\label{eq:cbsaeq}
\end{equation}
The constraints in the right side of \eqref{eq:cbsaeq} are called {\em consensus constraints}. The consensus constraints have the following formulation by a matrix equation.
For $1 \le m \le M$, let $e_m$ denote the edge indexed by $m$. We write $e_m = \{k, q\}$ whenever $e_m$ connects the nodes indexed by $k$ and  $q$.
For $1 \le k \le K$ and $1 \le m \le M$,  $\mC_{mk}$ is the following $D \times D$ matrix
\begin{equation}
\mC_{mk} = c_{mk} \mI, \text{ where } \ c_{mk} =
\begin{cases}
1, & \text{if } e_m = \{k, q\} \text{ and } k < q;\\
-1, & \text{if } e_m = \{k, q\}  \text{ and } q < k;\\
0, & \text{otherwise}.
\end{cases}
\label{eq:matconsconst}
\end{equation}
Let $\mC$ denote the $DM \times DK$ block matrix with blocks $\{\mC_{mk}\}_{m = 1, k = 1}^{M, K}$ and let $
\bar{\mQ} = [\mQ_1^T, \dots, \mQ_K^T ]^T,$ then the consensus constraints can be formulated as $\boldsymbol{C} \bar{\mQ} = \boldsymbol{0}$.

The minimization problem of \eqref{eq:cbsaeq} is inseparable and thus hard to compute in a distributed setting. That is, one cannot find the exact solution by just computing and adding results from each node. Instead, one needs to invoke the dual problem, which we describe next. The Lagrangian for problem \eqref{eq:cbsaeq} is
\begin{equation*}
L(\blamb, \boldsymbol{\bar{Q}}) = \sum \limits_{k = 1}^K F_k(\mQ_k) + \tr(\blamb^T \mC \bar{\mQ}),
\end{equation*}
where $\blamb = [\blamb_1^T, \dots, \blamb_M^T ]^T \in \bbR^{MD \times D},$ and the dual function is
\begin{equation}
d(\blamb) = \min \limits_{\bar{\mQ} \in S^K} L(\blamb, \bar{\mQ}).
\label{eq:consslave}
\end{equation}
Finally, the dual problem of \eqref{eq:cbsaeq} is
\begin{equation}
\hat{\blamb} =  \argmax \limits_{\blamb \in \bbR^{M D \times D} } d(\blamb).
\label{eq:maxprob}
\end{equation}

Recall that strong duality means that the minimizer of \eqref{eq:consslave} with $\hat{\blamb}$ found by the dual problem \eqref{eq:maxprob} coincides with the minimizer of \eqref{eq:cbsaeq}. In order to solve \eqref{eq:consslave}, the CBGA procedure uses the following separability of the dual function: $d(\blamb) = \sum_{k = 1}^K d_k(\blamb),$ where
\begin{equation}
\label{eq:cbga_local_main}
d_k(\blamb) = \min_{\mQ_k \in S} (F_k(\mQ_k) +  \tr(\blamb_m^T \mA_k)),
\end{equation}
\begin{equation}
\label{eq:def_Ak}
\mA_k = \sum_{m \in \cE_k} c_{mk} \blamb_m^T,
\end{equation}
$\{c_{mk} \} _{m=1, k=1} ^ {M ~~~ K}$ are defined in \eqref{eq:matconsconst}
and $\cE_k$ denotes the set of all edges that contain the node $k.$ Such separation gives rise to a distributed solution of \eqref{eq:consslave}. In order to solve \eqref{eq:maxprob}, the CBGA procedure applies subgradient descent over $\blamb$. According to \cite{dbsa}, one possible subgradient is $\mC \boldsymbol{\bar{Q} (\blamb)},$ where $\boldsymbol{\bar{Q}}(\blamb)$ is the solution of \eqref{eq:consslave} for the given $\blamb.$ Moreover, if $d(\blamb)$ is differentiable, then $\mC \boldsymbol{\bar{Q} (\blamb)}$ is the gradient. Therefore, the CBGA algorithm simultaneously solves problems \eqref{eq:consslave} and \eqref{eq:maxprob}. It starts with an initial guess of $\blamb$, then solves the separable problem of \eqref{eq:consslave}, next uses it for subgradient descent update of \eqref{eq:maxprob}, which results in a new value of $\blamb$, and iterates the two main steps until convergence.
% for the given $\blamb$.
The CBGA procedure converges if the following conditions are satisfied (see \cite{dbsa}): 1. the set $H$ is convex and the functions $F_k$ are convex; 2. strong duality holds for \eqref{eq:cbsaeq}; 3. the subgradients of $d(\blamb)$ are uniformly bounded for all values of $\blamb$.
We emphasize that this procedure assumes a solution of the separable problem in \eqref{eq:consslave} and without such a solution it is inapplicable.

%----------------------------------------------------------------------------------------

\section{Distributed GMS}
\label{sec:distgms}

We review the GMS problem in \S\ref{sec:gms_review}, propose a distributed solution in \S\ref{sec:cbgagms}, establish convergence guarantees in \S\ref{sec:cbga_gms_conv} and discuss the time complexity and possible reduction of the communication cost in \S\ref{sec:timecomplex}.

\subsection{Review of GMS}
\label{sec:gms_review}

In order to motivate the GMS algorithm for RSR, we first review the following convex formulation of PCA for full-rank data due to~\cite{gms}.
Assume that $\cX = \{ \vx_i\}_{i = 1}^N$ is a dataset of $N$ points in $\bbR^D,$ centered at $\v0$ and recall that the PCA $d$-subspace is the $d$-dimensional linear subspace minimizing the sum of squared residuals. If the dataset $\cX$ is full rank, then according to Theorem 10 of \cite{gms} the PCA $d$-subspace is spanned by the bottom $d$ eigenvectors of the following matrix $\hat{\mQ}$ (or equivalently, the top $d$ eigenvectors of $-\hat{\mQ}$):
\begin{equation}
\label{eq:pcapropmin}
\hat{\mQ} = \argmin \limits_{\mQ \in \bbH} \sum \limits_{\vx \in \cX} \Vert\mQ \vx\Vert^2, \text{ where } \bbH = \{\mQ \in \cS, \tr(\mQ) = 1 \}.
\end{equation}
Here and throughout the paper $\cS$ denotes the set of $D$-dimensional symmetric matrices, $\cSp$ denotes the set of $D$-dimensional positive semi-definite matrices and $\cSpp$ denotes the set of $D$-dimensional positive definite matrices.

The GMS procedure modifies~\eqref{eq:pcapropmin} by replacing the squared deviations $\Vert\mQ \vx\Vert^2$ in \eqref{eq:pcapropmin} with the more robust unsquared deviations $\Vert\mQ \vx\Vert$, while smoothing the resulted objective function around $\v0$ with a parameter $\delta > 0$. The convex minimization problem of GMS~\cite{gms} for the dataset
$\cX = \{\vx_i\}_{i=1}^N \subset \mathbb{R}^D$ and the regularization parameter $\delta$ is
\begin{equation}
\tilde{\mQ} = \argmin \limits_{\mQ \in \bbH} F^{\delta}(\mQ),
\label{eq:gmsstat}
\end{equation}
where $\bbH$ is defined in \eqref{eq:pcapropmin} and
\begin{equation}
F^{\delta}(\mQ) = \sum \limits_{\vx \in \cX, \Vert\mQ \vx \Vert \ge \delta} \Vert\mQ \vx\Vert + \sum \limits_{\vx \in \cX, \Vert\mQ \vx\Vert < \delta} \left( \frac{\Vert\mQ \vx \Vert^2}{2 \delta} + \frac{\delta}{2} \right).
\label{eq:dgmslocopt}
\end{equation}
Given a target dimension $1 \leq d \leq D-1$, the output of GMS is a $d$-dimensional subspace spanned by the bottom $d$ eigenvectors of $\tilde{\mQ}$ (or the top ones of $-\tilde{\mQ}$).

Clearly, the objective function in~\eqref{eq:pcapropmin} is strictly convex for full-rank data. The objective function in~\eqref{eq:dgmslocopt} is strictly convex under the following stronger condition, which is referred to as the two-subspaces criterion~\cite{gms}:
\begin{mydef}%[The two-subspaces criterion]
\label{def:twosubdef}
A dataset $\cY$ satisfies the two-subspaces criterion if
\begin{equation}
(\cY \cap \mL_1) \cup (\cY \cap \mL_2) \ne \cY \text{ for all } D-1 \text{ dimensional subspaces } \mL_1, \mL_2 \in \bbR^D.
\label{eq:twosubeq}
\end{equation}
\end{mydef}
When this criterion is satisfied, the unique minimizer of~\eqref{eq:gmsstat} can be computed by a very simple IRLS procedure (see Algorithm 2 in \cite{gms}).
If the dataset is not centered, one may appropriately center it at each iteration of the IRLS procedure. Alternatively and more commonly, one may initially center the original data by the geometric median.

Zhang and Lerman~\cite{gms} discuss the conditions under which GMS recovers the underlying subspace and show that they hold with high probability under a certain probabilistic model describing inliers and outliers (see \S 1.3 and \S 2 of \cite{gms}).
These conditions can be non-technically described as follows. First, the inliers need to spread throughout the whole underlying subspace, that is, they cannot concentrate on a lower dimensional subspace of the underlying subspace. Second, the outliers need to spread throughout the complement of the underlying subspace within the ambient space. Third, the magnitude of outliers needs to be restricted and they may not concentrate around lines. Zhang and Lerman~\cite{gms} propose some ways of preprocessing the data to avoid some restrictions imposed by these conditions (see \S5.2 of \cite{gms}).

The GMS solution to \eqref{eq:dgmslocopt} can be interpreted as a robust inverse covariance estimator. Indeed, the solution
to the least-squares problem~\eqref{eq:pcapropmin} is a scaled version of the inverse sample covariance (see Theorem~10 of \cite{gms}).
The IRLS procedure, which aims to solve  \eqref{eq:dgmslocopt}, scales the cross products of the sample covariance at each iteration in a way which may avoid the effect of outliers, and then inverts the resulting matrix or a regularized version of it.
%However, this matrix can be ill-conditioned in some cases. %Such cases include sufficiently high percentages of inliers with small amount of noise.

%These solutions may be hard to automate and may not work in some cases, however, as discussed later, the distributed implementation of GMS naturally proposes a practical solution to these restrictions.

%-----------------------------------------------------

\subsection{Consensus-Based Subgradient Algorithm for Distributed GMS}
\label{sec:cbgagms}
We assume a dataset $\cX$ with $\{\cX_k \}_{k=1}^K$ distributed at $K$ nodes. We further assume that for $1 \le k \le K,$ $\cX_k$ satisfies the two-subspaces criterion (see Definition~\eqref{def:twosubdef}), so they are full rank.
For general $\cX_1, \dots, \cX_K$ which may not satisfy this criterion, we suggest reducing their dimensions (see e.g., the discussion in \S\ref{sec:distpca}) so that they are full-rank.
In typical cases of noisy inliers concentrated around a subspace, the preprocessed $\cX_1, \dots, \cX_K$ with full rank will also satisfy the two-subspaces criterion.

We follow \S\ref{sec:cbga} and solve the minimization problem for the dual function of GMS in each node, while communicating these solutions via CBGA. Following \eqref{eq:cbga_local_main}, \eqref{eq:gmsstat} and \eqref{eq:dgmslocopt}, we need to solve at each node the following optimization problem:
\begin{equation}
d_k(\blamb) = \min \limits _{\mQ \in \bbH} G_k^{\delta} (\mQ) \text{ for } G_k^{\delta} (\mQ) = F_k^{\delta} (\mQ) + \tr(\mQ \mA_k),
\label{eq:gmsslavegen}
\end{equation}
where
%$\mA_k = \sum_{m \in \cE_k} c_{mk} \blamb_m^T,$ $\{c_{mk} \} _{m=1, k=1} ^ {M ~~~ K}$ are defined in \eqref{eq:matconsconst} and
\begin{equation*}
F_{k}^{\delta}(\mQ) = \sum \limits_{\vx \in \cX_k, \Vert\mQ \vx\Vert \ge \delta} \Vert\mQ \vx\Vert + \sum \limits_{\vx \in \cX_k, \Vert\mQ \vx\Vert < \delta} \left( \frac{\Vert\mQ \vx\Vert^2}{2 \delta} + \frac{\delta}{2} \right).
\end{equation*}
To find the minimizer of \eqref{eq:gmsslavegen} sufficiently fast, we introduce an iterative algorithm similar to Algorithm 2 of \cite{gms} and guarantee its $r$-linear convergence. Let  $\mQ_k^0 = \mI/D$ (or arbitrarily fix $\mQ_k^0 \in \cSpp \cap \bbH$) and for iteration $1 \le t \le T,$ let $\mQ_k^t$ be the solution of the following Lyapunov equation in $\mQ$,
where $c_k \in \bbR$ is chosen such that $\tr(\mQ_k^t) = 1$:
\begin{equation}
\mQ  \left( \sum \limits_{\vx \in \cX_k} \frac{\vx \vx^T}{2 \max(\Vert\mQ_k^{t-1} \vx\Vert, \delta)} \right) + \left( \sum \limits_{\vx \in \cX_k} \frac{\vx \vx^T}{2 \max (\Vert\mQ_k^{t-1} \vx\Vert, \delta)} \right) \mQ = c_k \mI - \mA_k.
\label{eq:itlyapsolreg}
\end{equation}

The following lemma establishes the existence and uniqueness of $c_k \in \bbR$ and $\mQ_k^t \in \cSpp \cap \bbH$, which satisfy \eqref{eq:itlyapsolreg}. It is proved in \S\ref{append:lem2}.
\begin{lemma}
Let $\cX = \{\vx_i\}_{i=1}^N$ be a full rank dataset in $\bbR^{D},$ $\mQ \in \cSpp \cap \bbH$ and $\mA \in \cS$ with $\tr(\mA) = 0$ and
\begin{equation}
\Vert \mA \Vert_2 \le 1 \biggr/ \tr \left( \left(\sum \limits _{\vx \in \cX} \frac{\vx \vx^T}{2\max(\Vert\vx\Vert, \delta)} \right)^{-1}\right).
\label{eq:gmsthcond}
\end{equation}
There exists a unique $c' \in \bbR$ such that the following equation with $c=c'$
\begin{equation}
\mP \left( \sum \limits_{\vx \in \cX} \frac{\vx \vx^T}{2 \max(\Vert\mQ \vx\Vert, \delta)} \right) + \left( \sum \limits_{\vx \in \cX} \frac{\vx \vx^T}{2 \max(\Vert\mQ \vx\Vert, \delta)} \right) \mP + \mA = c \mI
\label{eq:lyappsd}
\end{equation}
has a unique solution $\mP \in \cSpp \cap \bbH$.

If $\mQ_*$ is the solution of \eqref{eq:lyappsd} with $c=0$ and $\mA = \mA_k$, then
\begin{equation}
\label{eq:findingc}
c' =  - {2(\tr(\mQ_*) - 1)}\biggr/{\tr \left(\left( \sum \limits _{\vx \in \cX} \frac{\vx \vx^T}{2 \max(\Vert\mQ_* \vx\Vert, \delta)} \right)^{-1} \right)}.
\end{equation}
\label{lem:lemmapsd}
\end{lemma}

\Cref{algo:solveequation1} summarizes the above procedure of solving \eqref{eq:gmsslavegen}.
In \S\ref{sec:cbga_gms_conv} we establish the $r$-linear convergence of $\{\mQ_k^t\}_{t \in \bbN}$ to the minimizer of \eqref{eq:gmsslavegen}.

Given this solution of the local problem, the iterative CBGA algorithm for GMS is straightforward. As explained in \S\ref{sec:cbga}, at each iteration $s \ge 1$ and edge $e_m = \{k, q\}$, indexed by $1 \leq m \leq M$, the CBGA algorithm needs to update the corresponding $\blamb_m^s$ by the following gradient descent procedure
\begin{equation}
	\blamb_m^{s} = \blamb_m^{s-1} + \mu \cdot (c_{mk} \mQ_k^{s-1} - c_{mk} \mQ_q^{s-1}).		
\label{eq:lambdupdate}
\end{equation}					
Note that the update of $\blamb_m^{s}$ in \eqref{eq:lambdupdate} uses $\blamb_m^{s-1}$ and the local solutions   $\{\mQ_k^{s-1}\}_{k=1}^K$ of the previous iteration $s-1$.
The idea is to use $\blamb_m^s$ in solving the local problems. However, these problems only require the matrices $\mA^{s}_k =  \sum \limits_{m \in \cE_k} c_{mk} (\blamb_m^{s})^T$ for $k=1,\ldots, K$.
The combination of \eqref{eq:lambdupdate}, the latter expression for $\mA^{s}_k$ (see also \eqref{eq:def_Ak}), the fact that $c_{mk}^2=1$ whenever the $m$th edge is incident to the $k$th vertex and appropriate replacement of the set of edges $\cE_k$ with the set of vertices $\cN_k$ results in the following update formula		
\begin{equation}
\label{eq:analog_cbga}
\mA_k^s = \mA_k^{s-1} + \rho \sum\limits_{q  \in \cN_k} \left(\mQ_k^{s-1} - \mQ_q^{s-1} \right).
\end{equation}
The CBGA procedure for GMS thus iteratively updates the matrices $\{\mA_k^s\}_{k=1}^K$, by using the solutions of the local problems according to \eqref{eq:analog_cbga}, and solves the local problems by using the matrices $\{\mA_k^s\}_{k=1}^K$.
This simple procedure, which we refer to as CBGA-GMS is summarized in \Cref{algo:consgms}. In \S\ref{append:lemscalecond} we discuss how a sufficiently small step-size in \Cref{algo:consgms} ensures that the above condition \eqref{eq:gmsthcond}, which is necessary for solving the local problems, is satisfied at each node for all iterations of \Cref{algo:solveequation1}. We also explain in \S\ref{append:lemscalecond} why the required upper bound in \eqref{eq:gmsthcond} can be relaxed in practice and based on this observation we suggest a practical choice for the step-size in \eqref{eq:step_size_practice}.

%Using this algorithm and CBGA, \Cref{algo:consgms} (CBGA-GMS) distributedly solves the GMS problem. In \S\ref{append:lemscalecond} we discuss how a sufficiently small step-size in \Cref{algo:consgms} ensures that the above condition \eqref{eq:gmsthcond} is satisfied at each node for all iterations of \Cref{algo:solveequation1}. We also explain in \S\ref{append:lemscalecond} why the required upper bound in \eqref{eq:gmsthcond} can be relaxed in practice and based on this observation we suggest a practical choice for the step-size in \eqref{eq:step_size_practice}.

\begin{algorithm}[htbp]
\caption{Algorithm for computing the minimizer of \eqref{eq:gmsslavegen}}
\label{algo:solveequation1}

\begin{algorithmic}
 \STATE \textbf{Input:} $\cX = \{ \vx_1, \dots, \vx_N\} \subseteq \mathbb{R}^D$: data, $\mA_k \in \cS$ with $\tr(\mA_k) = 0,$ $T_{GMS}$: stopping iteration number, $\delta:$ regularization parameter (default: $10^{-10}$)
%\STATE \textbf{Output:} $\hat{\mQ}_k \in \bbH$
\STATE \textbf{Set: } $\mQ_k^0 = \mI / D$ and $t = 0$

\WHILE {$t \le T_{GMS}$ or $G_k^{\delta}(\mQ_k^{t+1}) > G_k^{\delta}(\mQ_k^{t})$}
	\STATE
	\begin{itemize}
		\item Let $\mQ_*$ be the solution of \eqref{eq:lyappsd} with $\mQ = \mQ_k^t,$ $c = 0$ and $\mA = \mA_k$
		\item Compute $c'$ according to \eqref{eq:findingc}
		\item Let $\hat{\mQ}_k^{t+1}$ be the solution of \eqref{eq:lyappsd} with $\mQ = \mQ_k^t,$ $c=c'$ and $\mA = \mA_k$
		\item $t := t+1$
	\end{itemize}
\ENDWHILE

\RETURN $\hat{\mQ}_k := \mQ_k^t$

\end{algorithmic}
\end{algorithm}

\begin{algorithm}[htbp]
\caption{Consensus-Based Subgradient Algorithm for GMS (CBGA-GMS)}
\label{algo:consgms}

\begin{algorithmic}
 \STATE \textbf{Input:} Network with $K$ nodes and $M$ edges, $\cX_1, \dots, \cX_K:$ datasets in the $K$ nodes, $T_{CBGA}, T_{GMS}$: stopping iteration numbers, $\delta$: regularization parameter (default: $10^{-10}$) and $\mu$: sufficiently small constant step-size %satisfying

%\STATE \textbf{Output:} $L_1, \dots, L_K: d$-dimensional subspaces for each node

\STATE \textbf{Set:} For all $1 \leq m \leq M$, $\blamb_m^0 = \v0$  and for all $1 \leq k \leq K$, $\mA_k^0 = \v0$  and $\mQ_k^0$ is the solution of \Cref{algo:solveequation1} with input $\cX_k$, $\mA_k^0$, $T_{GMS}$ and $\delta$
\FOR{$s = 1:T_{CBGA}$}
	\FOR{$k = 1:K$}
	\STATE
	\begin{itemize}
		\item Transmit $\mQ_k^{s-1}$ to $\cN_k$
		\item Compute $\mA_k^s$ according to \eqref{eq:analog_cbga}	
		\item  $\mQ_k^{s}$ is the output of \Cref{algo:solveequation1} with input $\cX_k, \mA_k^{s}, T_{GMS}$ and $\delta$
	\end{itemize}
	\ENDFOR
\ENDFOR

\RETURN $L_k := $ the span of the bottom $d$ eigenvectors of $\mQ_k^{T_{CBGA}}, 1 \le k \le K$
\end{algorithmic}
\end{algorithm}

%----------------------------------------------------------------------------------------

\subsection{Properties of CBGA-GMS}
\label{sec:cbga_gms_conv}

We establish $r$-linear convergence of \Cref{algo:solveequation1} and briefly discuss the mere convergence of \Cref{algo:consgms} and its recovery guarantees.
For completeness, we include the definition of $r$-linear convergence.
\begin{definition}
\label{def:rlinear}
A sequence $\{x_k\}_{k=1}^{\infty} \subset \bbR$ $r$-linearly converges to $x$ if there exists a sequence $\{v_k\}_{k=1}^{\infty} \subset \bbR,$ such that $|x_k - x| < v_k$ for all $k$ and there exists $q \in (0, 1)$ such that $v_{k+1} \le q v_k$ for all $k$ sufficiently large.
\end{definition}

The following theorem guarantees that $\{\mQ_k^t\}_{t \in \bbN}$ of \Cref{algo:solveequation1} $r$-linearly converges to the unique minimizer of \eqref{eq:gmsslavegen}. This theorem is later proved in \S\ref{sec:thproof}.
\begin{Theorem}
Assume $\cX_k = \{\vx_i\}_{i=1}^{N_k} \subset \mathbb{R}^D$ satisfies the two-subspaces criterion, $\mA_k \in \cS$ satisfies \eqref{eq:gmsthcond} and $\tr(\mA_k) = 0$.
If $\{\mQ_k^t\}_{t \in \bbN}$ is obtained by \Cref{algo:solveequation1} at node $k$ with $T_{GMS} = \infty$, then it %$\{\mQ_k^t\}_{t \in \bbN}$
$r$-linearly converges to the unique minimizer of \eqref{eq:gmsslavegen}.
\label{th:itgmsthreg}
\end{Theorem}

Note that CBGA-GMS is a gradient descent method. Indeed, Theorem 2 of \cite{gms} implies the strict convexity of $F^{\delta}$.
This and Theorems 26.1 and 26.3 of \cite{Rockafellar70convexanalysis} imply the differentiability of its dual function $d(\blamb) = \sum_{k = 1}^K d_k(\blamb)$, where $d_k(\blamb)$ is defined in \eqref{eq:gmsslavegen}.

The conditions for convergence of CBGA discussed in \S\ref{sec:cbga} are satisfied for CBGA-GMS. Indeed, the first condition is straightforward, since $G_k^{\delta}$ and $\bbH$ are convex. The strong duality of the problem is shown by easily verifying Slater's condition (see \S 5.2.3 of \cite{boydslater}). Finally,
the gradient of $d(\blamb)$ is $C  \bar{\mQ}$ and its norm is bounded by $K ||C||.$ Indeed, for each $1 \le k \le K,$ the $k$th block of $\bar{\mQ}$, $\mQ_k$, is in $\cSpp$ with $\tr (\mQ_k) = 1$ and thus $||C  \bar{\mQ}|| \le K ||C||.$
%the subgradient of $d(\blamb)$ is $\mQ_k$ or $-\mQ_k$ and since $\tr(\mQ_k) = 1$ for $1 \le k \le K$ it follows that $\Vert\mQ_k\Vert_2 \le 1$, which verifies the last condition.

Since the convex optimization problem for the total data of CBGA-GMS is the same as the convex optimization problem for regular GMS \cite{gms}, the exact and near recovery theory of CBGA-GMS follow from \cite{gms}.

%----------------------------------------------------------------------------------------

\subsection{Time Complexity}% and Possible Communication Cost Reduction}
\label{sec:timecomplex}

 \Cref{algo:solveequation1} solves  \eqref{eq:itlyapsolreg} twice.
%The complexity of CBGA-GMS depends on the complexity of solving \eqref{eq:itlyapsolreg}, which is called twice in each iteration of \Cref{algo:solveequation1}.
The computation of the coefficient of  \eqref{eq:itlyapsolreg}, $\sum_{i = 1}^{N_k} \vx_i \vx_i^T / (2 \max(\Vert\hat{\mQ}_k^{s-1} \vx_i\Vert, \delta))$, requires  $O \left( N_k \times D^2 \right)$ operations. Solving   \eqref{eq:itlyapsolreg} requires $O \left( D^3 \right)$
operations (see~\cite{bartels1972solution}). Since $N_k \geq  D,$ the total complexity for each iteration of \cref{algo:solveequation1} at node $k$ is $O \left( N_k \times D^2 \right).$ Denoting $N_{\max} = \max_{1 \le k \le K} {N_k},$ we conclude that the  complexities of Algorithms \ref{algo:solveequation1} and  \ref{algo:consgms} are $O \left( T_{GMS} \times N_{\max} \times D^2 \right)$ and $O \left( T_{CBGA} \times T_{GMS} \times N_{\max} \times D^2 \right)$ respectively.

\Cref{algo:consgms} transfers $D \times D$ matrices between nodes in each iteration, which might not be cost efficient. In order to reduce the communication cost we suggest transferring only the top $d$ eigenvectors of those matrices. Once a node receives the top $d$ eigenvectors, it reconstructs the $D \times D$ matrix $\mU^T \mU / \tr(\mU^T \mU),$  where $\mU \in \bbR^{d \times D}$ contains the orthogonal top $d$ eigenvectors as rows. We cannot guarantee the convergence of this modified procedure, but it seems to work well in practice.

%----------------------------------------------------------------------------------------

%\subsection{Consensus ADMM algorithm for GMS}
%\label{sec:cadmmgms}

%----------------------------------------------------------------------------------------

\section{Distributed Reaper and Distributed FMS}
\label{sec:distreaperfms}

We present distributed versions of two other RSR algorithms: Reaper \cite{reaper} and FMS \cite{maunu2014fast}. These algorithms are reviewed in \S\ref{sec:review_repaer_fms} and their straightforward distributed implementations are explained in \S\ref{sec:implement_dist_reaper_fms}.

%They are obtained by directly applying a distributed PCA algorithm. %Distributed reaper uses all principal components and distributed FMS requires only the top ones.

\subsection{Review of the Reaper and FMS Algorithms}
\label{sec:review_repaer_fms}
Assume a dataset $\cX \subset \mathbb{R}^D$, a target dimension $d \in \{1, 2, \dots D-1\}$ and a regularization parameter $\delta > 0$.

The Reaper algorithm \cite{reaper} solves the following convex optimization problem\footnote{The formulation in \cite{reaper} adds the additional optimization constraint $\mI - \mP \in \cSp$, but as is obvious from the proof of Lemma 14 in \cite{gms}, it is not needed and thus omitted from \eqref{eq:reapermin}}:
\begin{equation}
\min \limits_{{\mP \in \cSp, \ \tr(\mP) = d}} \sum_{\substack{\vx \in \cX \\ \Vert \vx - \mP \vx \Vert \ge \delta}} \Vert \vx - \mP \vx\Vert + \sum \limits_{\substack{\vx \in \cX \\ \Vert\vx - \mP\vx\Vert < \delta}} \left( \frac{\Vert\vx - \mP \vx\Vert^2}{2 \delta} + \frac{\delta}{2} \right).
\label{eq:reapermin}
\end{equation}
It uses an IRLS framework for  minimizing \eqref{eq:reapermin}. %, which is described in \cite{reaper}.
The robust $d$-subspace is spanned by the top $d$ eigenvectors of this solution.
A generic condition for subspace recovery by Reaper with an error bound is established in \cite{reaper}.\footnote{For simplicity, the analysis in \cite{reaper} is restricted to the case where $\delta = 0$.}  It requires similar restrictions as those described in the first and third non-technical conditions for GMS in \S\ref{sec:gms_review}.

Note that plugging $\mQ = \mI - \mP$ into \eqref{eq:dgmslocopt} results in an objective function similar to \eqref{eq:reapermin}. The main difference is that \eqref{eq:reapermin} further assumes that $\mP \in \cSp$. %Both of them can be seen as regularized convex relaxations of the following non-convex subspace least absolute deviations problem: minimize $\sum_{\vx \in \cX } \| \mQ \vx \|$ over all orthogonal projectors $\mQ$ of rank $D-d$.
%Note that the least squares minimization problem, which replaces $\| \mQ \vx \|$ with $\| \mQ \vx \|^2$ leads to the PCA subspace solution (more precisely, projector onto its orthogonal complement).
%As mentioned in \cite{reaper}, Reaper is the tightest convex relaxation of this problem, however, the solution of GMS can be seen as a robust version of the inverse covariance (see section 3.1 of~\cite{gms}) and has a simpler IRLS implementation.

The FMS algorithm \cite{maunu2014fast} tries to directly solve a regularized least unsquared deviations variant of PCA. Recall that the PCA subspace minimizes the least-squares function
$\sum_{\vx \in \cX} \dist^2(\vx, L),$ where $\dist(\vx, L) = \min_{y \in L} \Vert \vx - \vy \Vert_2,$ over the Grassmannian $G(D,d),$ which is the set of $d$-dimensional linear subspaces in $\bbR^D.$ The least unsquared deviations cost function is $\sum_{\vx \in \cX} \dist(x, L),$ where $L \in G(D,d).$ FMS aims to minimize the following smooth version of this function with the regularization parameter $\delta > 0$:
\begin{equation}
\label{eq:fmsobjective}
\min \limits_{L \in G(d, D)} \sum \limits_{\substack{\vx \in \cX, \dist(\vx, L) \ge \delta}} \dist(\vx, L) + \sum \limits_{\substack{\vx \in \cX, \dist(\vx, L) < \delta}} \left( \frac{\dist^{2}(\vx, L)} {2 \delta} + \frac{\delta}{2}\right).
\end{equation}
%among all $L \in G(D,d).$
This minimization is hard to solve in general (it was proved to be NP hard when $\delta=0$~\cite{clarkson2015input}). FMS is a straightforward IRLS heuristic for solving \eqref{eq:fmsobjective}. At each iteration it scales the original data points by the square root of their distance to the subspace of the previous iteration and then computes the current subspace by applying PCA to the scaled data.
Recovery and $r$-linear convergence of FMS were established only for data generated from very particular probabilistic models~\cite{maunu2014fast} .
However, in practice FMS seems to obtain competitive accuracy and speed for many datasets.  %We remark that Reaper and GMS are convex relaxations of the FMS minimization problem, where Reaper is tighter. In fact, it is the tightest convex relaxation possible (see \S 1.4 of \cite{reaper}).

We note that the target function in \eqref{eq:fmsobjective} is similar to that in \eqref{eq:dgmslocopt}, where $\dist(x,L)$ replaces $\Vert \mQ \vx \Vert$.
In fact, both GMS and Reaper are convex relaxations of the minimization in \eqref{eq:fmsobjective}, where Reaper is the tightest possible one~\cite{reaper}.
%Nevertheless, the solution of GMS can be seen as a robust version of the inverse covariance (see section 3.1 of~\cite{gms}) and has a simpler IRLS implementation.
%The algorithm starts with $L_0,$ the PCA subspace of $\cX,$ and at each iteration $s \geq 1$, it computes the $d$-dimensional PCA subspace $L_s$ of the scaled dataset: \begin{equation*} \{\vx_i/\max(\dist(\vx_i,L_{s-1}),\sqrt{\delta})\}_{i=1}^N.\end{equation*} This procedure iterates until convergence. The convergence analysis is presented in \cite{maunu2014fast}.
%The implementation of FMS avoids computing the eigenvalue decomposition of $D \times D$ matrices needed in the IRLS procedures of GMS and Reaper, and instead computes only SVD of $N \times D$ matrices.

\subsection{Distributed Implementations for Reaper and FMS}
\label{sec:implement_dist_reaper_fms}

We assume a dataset $\cX$ with $\{\cX_k \}_{k=1}^K$ distributed at $K$ nodes so that $\cX_k$ has full rank for $1 \le k \le K$. If the data is not full rank, it is preprocessed according to the discussion in \S\ref{sec:distpca}.

Distributed Reaper requires distributedly solving \eqref{eq:reapermin}. This can be done by applying distributed full PCA at each IRLS iteration of Algorithm 4.1 of \cite{reaper}. More precisely, this procedure first initializes the IRLS weights by $\beta_{\vx}^0 = 1$ for all data points $\vx \in \cX$. % (it can actually start with arbitrary $\beta_{\vx} > 0$ for $\vx \in \cX$).
Then, at each iteration $s \geq 1$ it applies distributed full PCA of the weighted dataset $\{\sqrt{\beta_{\vx}^{s-1}} \vx \}_{\vx \in \cX}$ to obtain $\mP_k^s$ at each processor with index $k.$ Then, it updates the weights by $\beta_{\vx} \leftarrow 1/\max(\delta, \Vert\vx - \mP_k^s \vx \Vert),$ for all $\vx \in \cX.$ This procedure is iterated until convergence and the local subspace is obtained by the top $d$ eigenvectors of $\mP_k^{s'}$, where $s'$ corresponds to the final iteration.

The distributed FMS is obtained by distributed PCA at each iteration of FMS. Note that FMS uses randomized SVD to find only the top $d$ principal components. For central processing and $D \gg d$, we recommend applying a distributed randomized SVD algorithm~\cite{halko2011finding}. For an ad hoc network, we are not aware of effective implementation of a distributed algorithm that can find only the top $d$ principal components.

%----------------------------------------------------------------------------------------

\section{Numerical Experiments}
\label{sec:numerical}

This section tests the distributed algorithms proposed in this paper using both synthetic and real data. It is organized as follows: \S\ref{sec:rsrmodel} describes the synthetic data model, \S\ref{sec:synthdata} contains experiments on data generated from this model and \S\ref{sec:realdata} contains experiments on real datasets.

Throughout this section, \Cref{algo:solveequation1} uses $T_{GMS} = 30$ and \Cref{algo:consgms} uses $T_{CBGA} = 250$ and $\mu$ as in \eqref{eq:step_size_practice} or in a specified range of values. In all RSR algorithms the regularization parameter is $\delta = 10^{-10}$. CBGA-PCA of \S\ref{sec:distpca} is used as ``distributed PCA'' and is also implemented in the iterative schemes of distributed Reaper and FMS. All codes necessary to duplicate these results are available in \url{https://github.com/vahanhuroyan/Distributed-RSR}.

\subsection{Synthetic Data Model for Distributed RSR}
\label{sec:rsrmodel}

In \S\ref{sec:synthdata} we use the following synthetic model to generate distributed RSR data. It depends on the following parameters: $K, N^0, N^1, D, d$ and $\sigma.$ We create a connected graph with $K$ nodes as explained below, and we randomly fix $L \in G(D, d).$ For each node we sample $N^1 / K$ inliers from the $d$-dimensional Multivariate Normal distribution $N(\v0, \mP_L),$ where $\mP_L$ denotes the orthoprojector onto $L,$ with additive Gaussian noise $N(\v0, \sigma^2\mI)$, where $0 \le \sigma < 1.$ Furthermore, for each node we sample $N^0 / K$ outliers from the uniform distribution on $[0, 1]^D.$ Note that the outliers are asymmetric. Unless otherwise specified (see \S\ref{sec:topconv}), the graph is obtained by arbitrarily generating a spanning tree with $K$ nodes and then randomly and independently connecting $2$ nodes with probability $1/2$. It is demonstrated for $K=8$ in Fig.~\ref{fig:netgr4}.

%----------------------------------------------------------------------------------------

\subsection{Demonstration on Synthetic Data}
\label{sec:synthdata}

We study the effect of the network topology and the step-size on the convergence rate of CBGA-GMS  in \S\ref{sec:topconv} and \S\ref{sec:convssize} respectively. In \S\ref{sec:cadmmvscbga} we compare the accuracy of a CADMM version of GMS with CBGA-GMS. In \S\ref{sec:compallmethods} we compare our proposed distributed RSR algorithms. In each experiment $50$ random samples are generated according to the model of \S\ref{sec:rsrmodel}. {The recovery error of the tested algorithm is averaged over the random $50$ samples. For Figs.~\ref{fig:topvsiter1}-\ref{fig:admmvsgr} we further average the recovery error over the $K$ processors to demonstrate the average rate of convergence. %For Figs.~\ref{fig:comp3}-\ref{fig:comp2} we only present the results for one processor, as the results are identical among all processors.}
We remark that in all experiments, the data is full rank at each processor, so there was no need to initially apply dimension reduction.

\subsubsection{The Influence of the Network Topology on Convergence}
\label{sec:topconv}
\begin{figure}[ht]
  \begin{subfigure}[b]{0.32\linewidth}
    \centering
    \begin{tikzpicture}[->,>=stealth',shorten >=1pt,auto,node distance=1.3cm, thick,main node/.style={circle,fill=white!20,draw,font=\sffamily\Large\bfseries}]

 	\node[main node] (1) {1};
 	\node[main node] (2) [right of=1] {2};
 	\node[main node] (3) [below right of=2] {3};
 	\node[main node] (4) [below of=3] {4};
 	\node[main node] (5) [below left of=4] {5};
 	\node[main node] (6) [left of=5] {6};
 	\node[main node] (7) [above left of=6] {7};
 	\node[main node] (8) [above of=7] {8};

  	\path[every node/.style={font=\sffamily\small}]
    	(1) edge [right] node[right] {} (2)
    		edge [right] node[right] {} (8)
    	(2) edge node [left] {} (1)
        	edge [right] node[right] {} (3)
    	(3) edge node [left] {} (2)
        	edge [right] node[right] {} (4)
    	(4) edge node [left] {} (3)
        	edge [right] node[right] {} (5)
    	(5) edge node [left] {} (4)
    	(6) edge [right] node {} (7)
   		(7) edge node [left] {} (6)
        	edge [right] node {} (8)
   		(8) edge node [left] {} (7);
   			edge node [left] {} (1)
	\end{tikzpicture}
	\caption{\text{ }}
  	\label{fig:netgr2}
  \end{subfigure}
  \begin{subfigure}[b]{0.32\linewidth}
    \centering
  	\begin{tikzpicture}[->,>=stealth',shorten >=1pt,auto,node distance=1.3cm, thick,main node/.style={circle,fill=white!20,draw,font=\sffamily\Large\bfseries}]

 	\node[main node] (1) {1};
 	\node[main node] (2) [right of=1] {2};
 	\node[main node] (3) [below right of=2] {3};
 	\node[main node] (4) [below of=3] {4};
 	\node[main node] (5) [below left of=4] {5};
 	\node[main node] (6) [left of=5] {6};
 	\node[main node] (7) [above left of=6] {7};
 	\node[main node] (8) [above of=7] {8};

  	\path[every node/.style={font=\sffamily\small}]
    	(1) edge [right] node[right] {} (2)
			edge [bend right] node[right] {} (3)
			edge [right] node[right] {} (4)
			edge [right] node[right] {} (5)
			edge [right] node[right] {} (6)
			edge [right] node[right] {} (7)
			edge [right] node[right] {} (8)
    	(2) edge node [left] {} (1)
			edge [right] node[right] {} (4)
			edge [right] node[right] {} (5)
			edge [right] node[right] {} (6)
			edge [right] node[right] {} (7)
			edge [right] node[right] {} (8)
    		edge [right] node[right] {} (3)
    	(3) edge node [left] {} (2)
			edge [bend left] node[right] {} (1)
			edge [right] node[right] {} (5)
			edge [right] node[right] {} (6)
			edge [right] node[right] {} (7)
			edge [right] node[right] {} (8)
    		edge [right] node[right] {} (4)
    	(4) edge node [left] {} (3)
			edge [right] node[right] {} (1)
			edge [right] node[right] {} (2)
			edge [right] node[right] {} (7)
			edge [right] node[right] {} (8)
			edge [bend right] node[right] {} (6)
    		edge [right] node[right] {} (5)
    	(5) edge node [left] {} (4)
			edge [right] node[right] {} (1)
			edge [right] node[right] {} (2)
			edge [right] node[right] {} (3)
			edge [right] node[right] {} (7)
			edge [right] node[right] {} (8)
    		edge [right] node {} (6)
    	(6) edge node [right] {} (5)
			edge [right] node[right] {} (1)
			edge [right] node[right] {} (2)
			edge [right] node[right] {} (3)
			edge [bend left] node[right] {} (4)
			edge [bend right] node[right] {} (8)
   			edge [right] node {} (7)
   		(7) edge node [left] {} (6)
			edge [right] node[right] {} (1)
			edge [right] node[right] {} (2)
			edge [right] node[right] {} (3)
			edge [right] node[right] {} (4)
			edge [right] node[right] {} (5)
    		edge [right] node {} (8)
   		(8) edge node [left] {} (7)
			edge [right] node[right] {} (1)
			edge [right] node[right] {} (2)
			edge [right] node[right] {} (3)
			edge [right] node[right] {} (4)
			edge [right] node[right] {} (5)
		edge [bend left] node[right] {} (6);
	\end{tikzpicture}
	\caption{\text{ }}
	\label{fig:netgr3}
	\end{subfigure}
  	\begin{subfigure}[b]{0.32\linewidth}
  	\centering
  	\begin{tikzpicture}[->,>=stealth',shorten >=1pt,auto,node distance=1.3cm,thick,main node/.style={circle,fill=white!20,draw,font=\sffamily\Large\bfseries}]

 	\node[main node] (1) {1};
 	\node[main node] (2) [right of=1] {2};
 	\node[main node] (3) [below right of=2] {3};
 	\node[main node] (4) [below of=3] {4};
 	\node[main node] (5) [below left of=4] {5};
 	\node[main node] (6) [left of=5] {6};
 	\node[main node] (7) [above left of=6] {7};
 	\node[main node] (8) [above of=7] {8};

  	\path[every node/.style={font=\sffamily\small}]
    	(1) edge [right] node[right] {} (2)
			edge [bend right] node[right] {} (3)
			edge [right] node[right] {} (7)
    	(2) edge node [left] {} (1)
			edge [right] node[right] {} (6)
			edge [right] node[right] {} (8)
    		edge [right] node[right] {} (3)
    	(3) edge node [left] {} (2)
			edge [bend left] node[right] {} (1)
    		edge [right] node[right] {} (4)
    	(4) edge node [left] {} (3)
			edge [bend right] node[right] {} (6)
			edge [right] node[right] {} (8)
   			edge [right] node[right] {} (5)
    	(5) edge node [left] {} (4)
        	edge [right] node {} (6)
    	(6) edge node [right] {} (5)
			edge [bend left] node[right] {} (4)
			edge [right] node[right] {} (2)
        	edge [right] node {} (7)
   		(7) edge node [left] {} (6)
			edge [right] node[right] {} (1)
    		edge [right] node {} (8)
   		(8) edge node [left] {} (7)
			edge [right] node[right] {} (2)
			edge [right] node[right] {} (4);
	\end{tikzpicture}
  	\caption{\text{ }}
  	\label{fig:netgr4}
	\end{subfigure}
	\caption{Three types of connected networks with 8 nodes. Fig. \ref{fig:netgr2}: sparsely connected network; Fig. \ref{fig:netgr3}: fully connected network; and Fig. \ref{fig:netgr4}: randomly connected network.}
\label{fig:netgraphs}
\end{figure}

To check the effect of the network topology on the convergence rate we use three different networks, whose graphs are shown in Fig.~\ref{fig:netgraphs}. The graph in Fig.~\ref{fig:netgr2} is sparse, the graph in Fig.~\ref{fig:netgr3} is fully connected and the graph in Fig.~\ref{fig:netgr4} is generated according to the recipe described in \S\ref{sec:rsrmodel}.
We generate data according to the model of \S\ref{sec:rsrmodel}, where $K = 10$, $N^1 = 200$, $N^0 = \num{2000}$, $D = 50$, $d = 3$, $\sigma = 0.1$ and $\mu = 100.$ The average recovery error as a function of the number of iterations for the $3$ different networks is shown in Fig.~\ref{fig:topvsiter1}. The fully connected network has the fastest convergence and as the network gets sparser, the convergence rate decreases.

\subsubsection{The Influence of the Step-size on the Convergence Rate}
\label{sec:convssize}

We generate data according to the model of \S\ref{sec:rsrmodel}, where $K = 10$, $N^1 = 200$, $N^0 = \num{2000},$ $D = 50$, $d = 3$, and $\sigma = 0.1$.
Fig.~\ref{fig:cbda1} shows the average recovery error for CBGA-GMS as a function of the number of iterations for $7$ different step-sizes: 10, 30, 50, 100, 150, 200 and the one proposed in \eqref{eq:step_size_practice}, whose value here is 22.5.
The average error of GMS for the total data is included as a baseline.
These results
imply that the convergence rate increases with the step-size. However, additional experiments, not reported in here, indicate that for a very large step-size the algorithm does not converge. We also note that for large step-sizes, the increase of the step-size does not  significantly change the convergence rate, for example, for step-sizes $150$ and $200$ we see almost the same result, while the difference between convergence results is obvious for smaller step-sizes.

\begin{figure}
\centering
\begin{subfigure}{.49\textwidth}
	\includegraphics[width=1.0\linewidth]{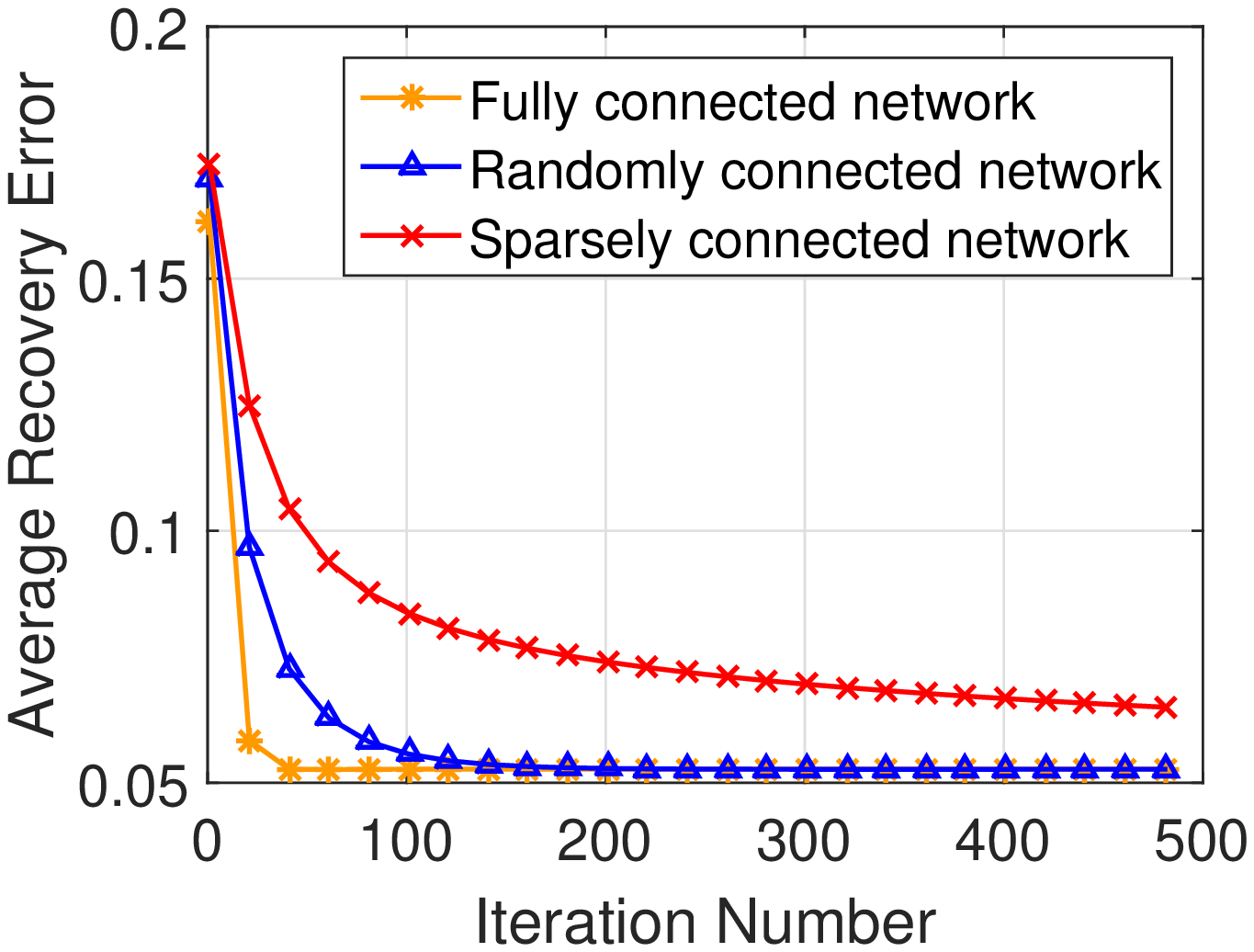}
	\caption{Influence of the network topology on the convergence rate of CBGA-GMS}
	\label{fig:topvsiter1}
\end{subfigure}
\begin{subfigure}{.49\textwidth}
	\includegraphics[width=1.0\linewidth]{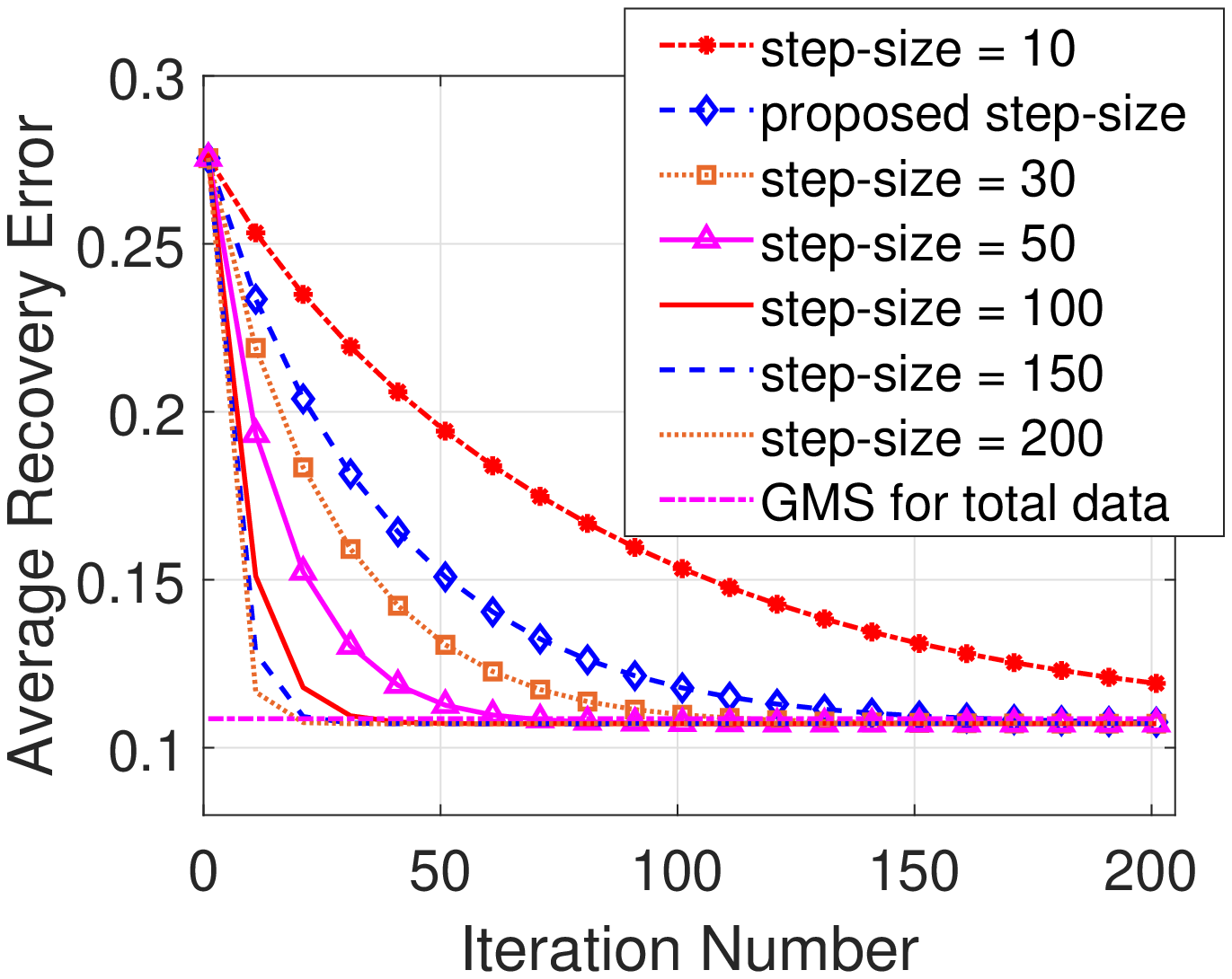}
	\caption{Influence of different step-sizes on the convergence rate of CBGA-GMS}
	\label{fig:cbda1}
\end{subfigure}

\begin{subfigure}{.49\textwidth}
		\includegraphics[width=1.0\linewidth]{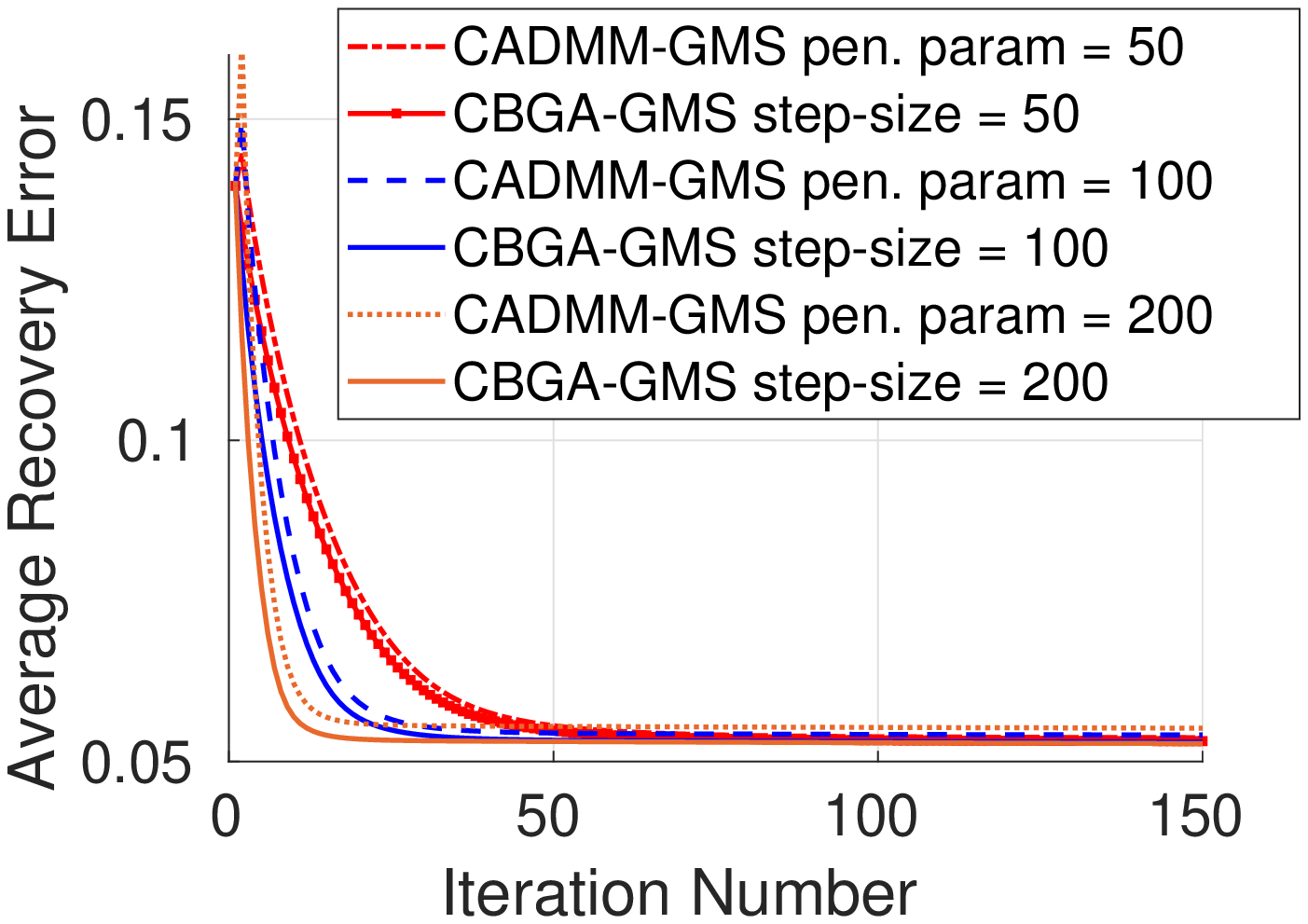}
	\caption{CBGA-GMS vs CADMM-GMS}
	\label{fig:admmvsgr}
\end{subfigure}
\begin{subfigure}{.49\textwidth}
	\includegraphics[width=1.0\linewidth]{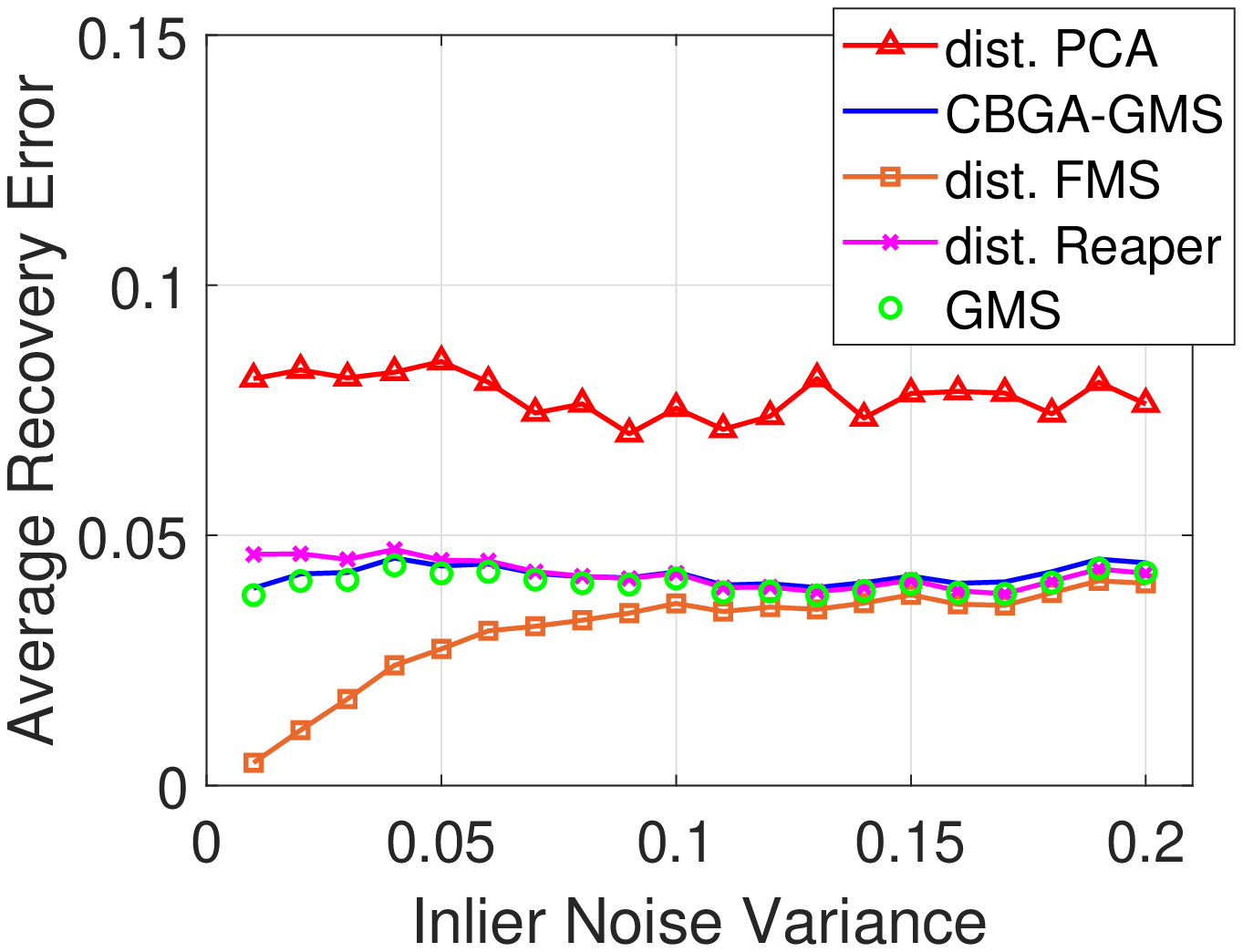}
	\caption{Influence of inlier noise variance on distributed PCA, Reaper, GMS and FMS}
	\label{fig:comp3}
\end{subfigure}

\begin{subfigure}{.49\textwidth}
	\includegraphics[width=1.0\linewidth]{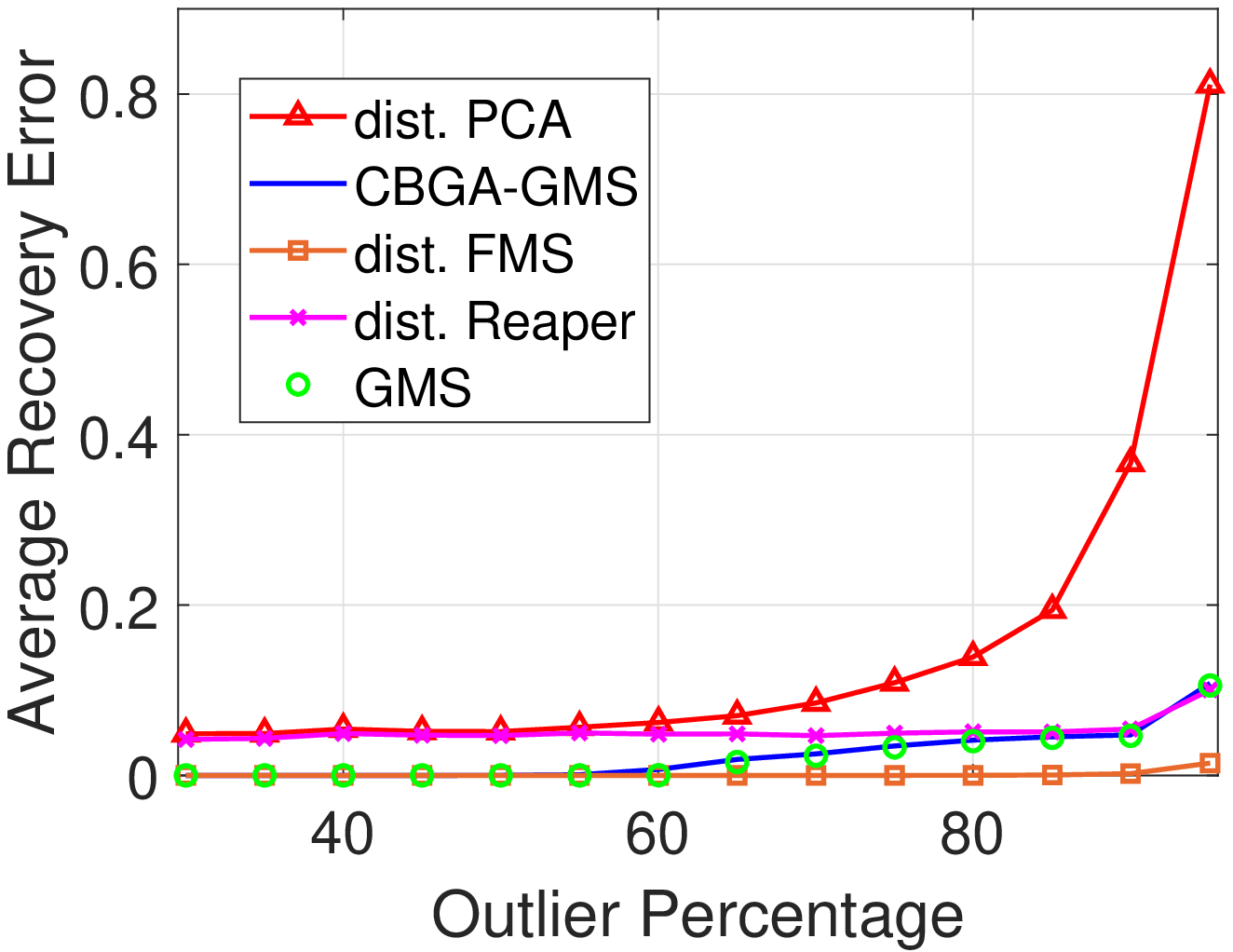}
	\caption{Influence of outlier percentage on distributed PCA, Reaper, GMS, FMS; $\sigma = 0$}
	\label{fig:comp1}
\end{subfigure}
\begin{subfigure}{.49\textwidth}
	\includegraphics[width=1.0\linewidth]{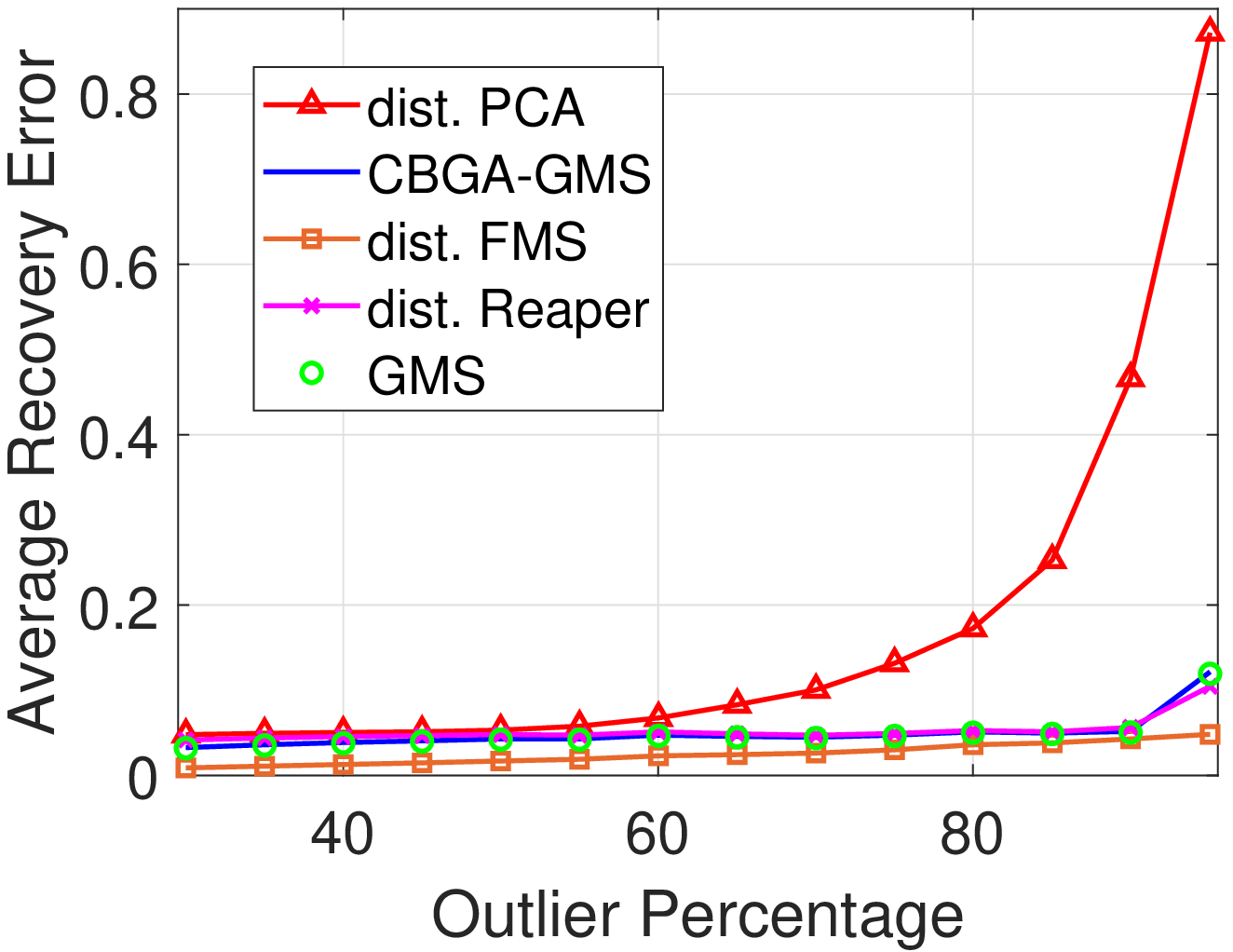}
	\caption{Influence of outlier percentage on distributed PCA, Reaper, GMS, FMS;  $\sigma = 0.05$}
	\label{fig:comp2}
\end{subfigure}
\caption{Demonstration of properties of the distributed algorithms on synthetic data.}
\end{figure}

%----------------------------------------------------------------------------------------

\subsubsection{Comparing CBGA-GMS with CADMM-GMS}
\label{sec:cadmmvscbga}
A CADMM scheme for GMS, which directly follows~\cite{mateos2010distributed}, is described in \S\ref{sec:cadmm}. It is referred to as CADMM-GMS.
Both  CBGA-GMS and CADMM-GMS are somewhat parallel and it follows from \eqref{eq:analog_cbga} and \eqref{eq:zks} that their corresponding parameters $\mu$ and $\rho$ play similar roles. We compare them using data generated from the model described in \S\ref{sec:rsrmodel}, where $K = 5$, $D = 50$,  $d = 3$, $\sigma = 0.05$, $N^0 = \num{5000}$ and $N^1 = 500$.
We tested the following same values of $\rho$ and $\mu$: 50, 100 and 200. We remark that the $\mu$ proposed in \eqref{eq:step_size_practice} obtained the value 51.1. Since both algorithms performed similarly when using this value and 50, we did not report the performance with this value.
Fig.~\ref{fig:admmvsgr} shows the recovery errors vs.~the number of iterations for both algorithms with these step-sizes. We note that both algorithms converge with very similar speed, where CBGA-GMS converges slightly faster. For the smaller values of $\mu$ and $\rho$ (100 and 200) the algorithms achieve the same recovery error. However, for the larger value of the parameter (300), the recovery error of  CADMM-GMS is slightly higher than the recovery error of CBGA-GMS.

%----------------------------------------------------------------------------------------

\subsubsection{Comparison of the Proposed Algorithms}
\label{sec:compallmethods}

We compare GMS, the $3$ proposed distributed RSR algorithms and distributed PCA in different settings and report the results in Figs.~\ref{fig:comp3}-\ref{fig:comp2}. Fig.~\ref{fig:comp3} demonstrates how the inlier noise variance $\sigma$ affects the convergence of the four methods. The data for this figure was created according to the model described in \S\ref{sec:rsrmodel}, where $D = 50,$ $d = 3,$ $K = 5,$ $N^0 = \num{3000}$, $N^1 = \num{1000}$ and $\sigma$ varies between $0$ and $0.2$ with increments of $0.01.$ In this figure, for all tested values of $\sigma,$ CBGA-PCA performs the worst and distributed FMS performs the best, where CBGA-GMS and distributed Reaper are somewhat comparable.
%performs slightly better than distributed Reaper if $\sigma \le 0.1,$ otherwise, distributed Reaper performs slightly better. Furthermore, as the inlier noise variance increases the error for distributed RSR methods also increases.

Figs.~\ref{fig:comp1} and \ref{fig:comp2} demonstrate the influence of the outlier percentage on the average recovery error of the four distributed methods and GMS (for the total data) with and without inlier noise. We generate data according to the model of \S\ref{sec:rsrmodel}, where $D = 50,$ $d = 3,$ $K = 10,$ $\sigma = 0$ for Fig.~\ref{fig:comp1}, $\sigma = 0.5$ for Fig.~\ref{fig:comp2}, $N^0 = \num{5000}$  and $N^1$ is chosen such that the outlier percentage in the total data varies between $30\%$ to $95\%$ with increments of $5\%$. For both cases ($\sigma = 0$ and $\sigma = 0.05$) and for all percentages of outliers, the recovery error for distributed FMS is the smallest one and that of CBGA-PCA is the largest one.
Figs.~\ref{fig:comp1} and \ref{fig:comp2} also demonstrate that when the data is corrupted with outliers (percentage of outliers higher than $65\%$), the distributed RSR algorithms perform significantly better than distributed PCA. For the case of $\sigma = 0$, %(inliers lie exactly on a $d$-dimensional subspace), %and the percentage of outliers at most $90\%,$
distributed FMS and CBGA-GMS succeed with exact recovery up to $90\%$ and $55\%$ of outliers respectively, whereas distributed Reaper could not exactly recover the subspace in the tested range.
% and in general was the worst performer among the distributed robust algorithms.

In Figs.~\ref{fig:comp3}, \ref{fig:comp1} and \ref{fig:comp2}, the recovery errors obtained by CBGA-GMS and GMS are comparable. We remark that the distributed implementations of PCA, Reaper and FMS also obtain similar recovery errors as the non-distributed ones in all of these experiments. However, since these figures are already dense, we do not report the results of the latter non-distributed implementations.

%----------------------------------------------------------------------------------------

\subsection{Real Data Experiments}
\label{sec:realdata}

Distributed RSR algorithms can be used as a preprocessing step for clustering, classification and regression. We apply our proposed distributed algorithms as a preprocessing step for two different tasks: linear regression, where we use the CTslices dataset $(N=\num{53500}, D=386)$~\cite{Lichman:2013}, and classification (multiclass SVM), where we use the Human Activity Recognition (HAR) dataset  $(N=\num{10299}, D=561)$~\cite{anguita2013public, Lichman:2013}.
%MNIST $(70000 \times 784)$ and NYTimes Bag of Words $(300000 \times 102660)$ to demonstrate k-means clustering,
For both datasets we apply initial centering by the geometric median and  to ensure full-rank data in all processors we reduce dimension to $D = 150$ by distributed exact PCA (see \S\ref{sec:distpca}). We remark that higher values of reduced dimensions $D$ were also possible. We report the results for one of the processors as they are the same for all of them.

\begin{figure}
\centering
\begin{subfigure}{.49\textwidth}
	\includegraphics[width=1.0\linewidth]{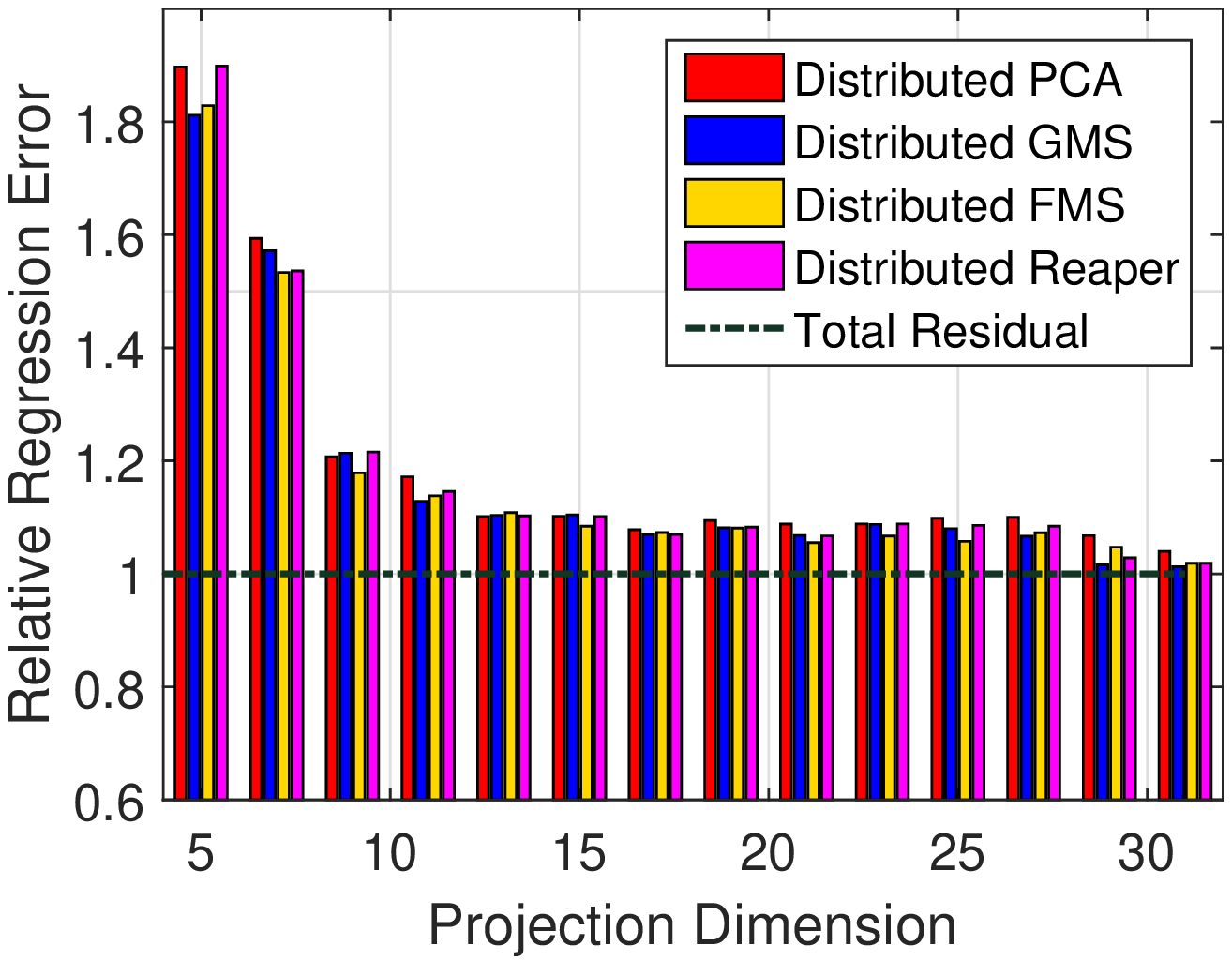}
	\caption{CTslices}
	\label{fig:ctslices}
\end{subfigure}
\begin{subfigure}{.49\textwidth}
	\includegraphics[width=1.0\linewidth]{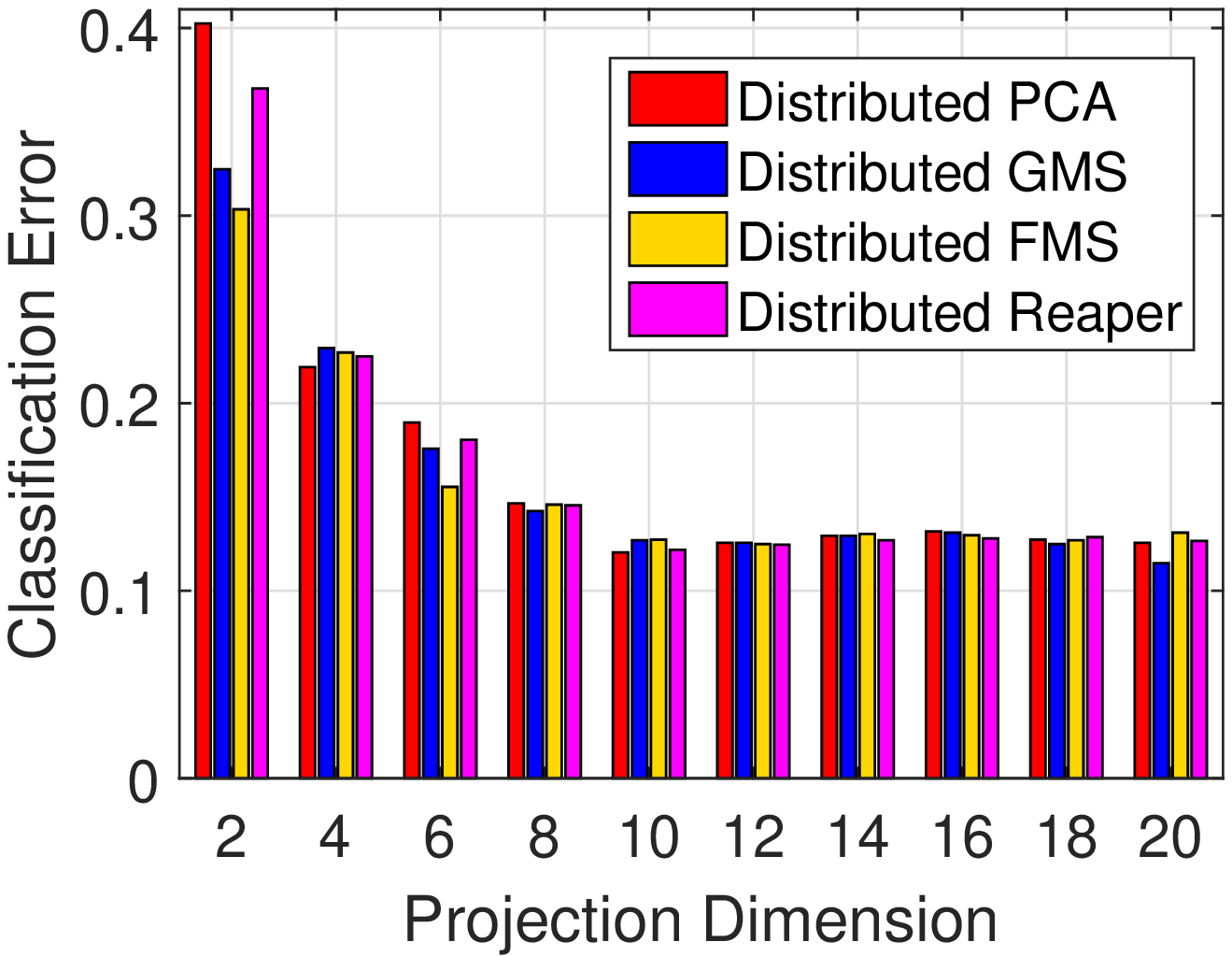}
	\caption{Human Activity Recognition (HAR)}
	\label{fig:hardata}
\end{subfigure}

\caption{Demonstration of the proposed distributed algorithms on two real datasets: CTslices and HAR.}
\end{figure}

%Fig.~\ref{fig:mnist} shows the results for MNIST. We divide the data between 10 processors each having $7000$ data points. We first mean center the data and then apply distributed PCA, GMS, FMS and Reaper to reduce the dimension of the data and apply K-means clustering in reduced dimension. As we can see in Fig.~\ref{fig:mnist} for small dimensions, distributed PCA has advantage over the robust methods, but as the dimension increases, distributed RSR methods perform better than distributed PCA.

For the CTslices data, the algorithms are trained on $\num{50000}$ data points and tested on $\num{3500}$ data points. The training data is divided between $5$ processors, each containing $\num{10000}$ data points. We apply CBGA-PCA, CBGA-GMS, distributed FMS and distributed Reaper to reduce the dimension of the dataset to lie between $5$ and $30.$ We then apply linear least squares regression in the reduced dimension. Fig.~\ref{fig:ctslices} reports the relative regression error for the different projected dimensions. The relative regression error is the regression error for the data with the reduced dimension divided by the relative error for the data in $150$ dimensions. %It demonstrates that %in general
%distributed RSR algorithms perform at least as good as distributed PCA.
We notice that for almost all dimensions, the relative errors of distributed FMS and GMS are lower than those of distributed PCA, and the relative errors of distributed Reaper are either lower or comparable to those of distributed PCA.

For the HAR data, the algorithms are trained on $\num{7352}$ data points and tested on $\num{2947}$ data points. The training data is divided between $8$ processors, each containing $919$ data points. We apply CBGA-PCA, CBGA-GMS, distributed FMS and distributed Reaper to reduce the dimension of the dataset to lie between $2$ and $20.$ We then apply classification in the reduced dimensions. Fig.~\ref{fig:hardata} reports classification error for the different projected dimensions. It demonstrates that in dimension $2,$ the distributed RSR algorithms, in particular, distributed FMS and GMS, have a clear advantage over distributed PCA. In other dimensions, distributed RSR algorithms perform at least as good as distributed PCA.

We comment that for all real datasets, the results of the distributed algorithms are very similar to those of the non-distributed ones. Differences between all distributed and non-distributed implementations may exist when the initial dimension $D$ is large and an initial dimension reduction by OSE is applied (see \S\ref{sec:distpca}). An effect of OSE on the performance of PCA in a distributed setting is documented in \cite{liang2014improved}.

%\S\ref{sec:distpca}

%Fig.~\ref{fig:bowdata} shows the results for BOW. We divide the data between $10$ processors, each containing $300000$ data points. Then we apply distributed PCA, GMS, FMS and Reaper algorithms to reduce the data dimension and then apply K-means clustering in reduced dimension. As we can see in Fig.~\ref{fig:bowdata} for dimensions $5 - 8$ distributed PCA performs better than the distributed RSR algorithms, for dimensions $9 - 14$ CBGA-GMS perform better than the other methods, and for larger dimensions PCA, GMS and Reaper perform comparable, with PCA and Reaper having slight advantage over GMS.

%----------------------------------------------------------------------------------------

\section*{Acknowledgements}

We thank Amit Singer for his helpful comments on the presentation of this work and his suggestion to demonstrate the ideas of CBGA-GMS on the problem of distributed computation of the geometric median.
This work was supported by NSF awards DMS-09-56072 and DMS-14-18386 and the Feinberg Foundation Visiting Faculty Program Fellowship of the Weizmann Institute of Science.

\appendix

%----------------------------------------------------------------------------------------

\section{Solutions of Related Problems}
\label{sec:demo}

We first use the idea of CBGA-GMS to solve two simpler problems: distributed computation of the PCA subspace and distributed computation of the geometric median. We then describe a CADMM solution for distributed GMS.

\subsection{Distributed PCA for Arbitrarily Distributed Network}
\label{sec:distpca}

Before describing the CBGA procedure for PCA, we remark that if the dimension $D$ is not high, then the following simple procedure can be applied to solve the problem. One may propagate the local covariance matrices among the network and recover the exact covariance matrix at each processor and use it for PCA computation.  We refer to it as exact distributed PCA. If the dimension $D$ is high,  then it can be reduced by an OSE procedure described below before applying the exact distributed PCA algorithm.

Our proposed CBGA-PCA algorithm is similar to \cite{aduroja2013distributed,valcarcel2010consensus,bertrand2011consensus}, but uses instead the PCA formulation in~\eqref{eq:pcapropmin}.
This formulation leads to a direct solution of the local optimization problem.
In order to apply~\eqref{eq:pcapropmin}, one needs to guarantee that for all $1 \leq k \leq K$, $\cX_k$ has full rank. If $\cX_k$ is rank-deficient, one can reduce its dimension. If the dimension $D$ is high, one can sample an Oblivious Subspace Embedding (OSE) matrix $\mH$ \cite{sarlos2006improved} and instead of $\cX_k$ consider $\mH \cX_k $. One common OSE $\mH$ has only one non-zero entry per row. By taking an appropriate number of rows for  $\mH$, one can assume that the projected data at each node has full rank. If the dimension $D$ is not high, then the exact distributed PCA, or a faster approximate version of it, can be used to reduce the dimension.

Next, we clarify the application of CBGA to \eqref{eq:pcapropmin}.
In view of \S\ref{sec:cbga}, it is sufficient to compute the dual function of \eqref{eq:pcapropmin} at each node, that is, compute for each $1 \leq k \leq K$:
%\begin{equation}
%d_k(\blamb) = \min \limits _{\mQ \in \bbH} \left( \sum \limits_{\vx \in \cX_k} \Vert\mQ \vx\Vert^2 + \sum \limits_{m \in \cE_k} c_{mk} \tr(\blamb_m^T \mQ) \right).
%\label{eq:dpca1}
%\end{equation}
%We can rewrite  \eqref{eq:dpca1} as
\begin{equation}
d_k(\blamb) = \min \limits _{\mQ \in \bbH} \left( \sum \limits_{\vx \in \cX_k} \Vert\mQ \vx\Vert^2 + \tr(\mA_k \mQ) \right), \text{ where } \mA_k = \sum \limits_{m \in \cE_k} c_{mk} \blamb_m^T.
\label{eq:dpca2}
\end{equation}
Appendix \ref{append:lem1proof_whole} guarantees the unique minimizer of \eqref{eq:dpca2} and explains how to find it.

Since the minimized function in \eqref{eq:pcapropmin} is strongly convex, it follows from \cite{kakade2012applications} that its dual function $d(\blamb) = \sum_{k = 1}^K d_k(\blamb)$, where $d_k(\blamb)$ is defined in \eqref{eq:dpca2}, is Lipschitz smooth. This implies that the CBGA algorithm for PCA converges to the PCA solution for the total data with rate $O(1/t).$ The complexity of CBGA-PCA is $O(T_{CBGA} \times N_{\max} \times D^2)$ (see \S\ref{append:find_min}).
This algorithm is not optimal in terms of complexity and communication. Indeed, the distributed exact PCA algorithm described above is simpler and achieves the exact PCA subspace. Nevertheless, we find this CBGA-PCA interesting for two reasons. First of all, it is similar to previous attempts \cite{aduroja2013distributed,valcarcel2010consensus,bertrand2011consensus} that did not clarify how to solve the local dual problem. Second of all, CBGA-PCA simply demonstrates the main idea of  the more complicated CBGA-GMS procedure.
%For the latter procedure, where one cannot effectively propagate between nodes a robust local version of the covariance.

%----------------------------------------------------------------------------------------

\subsection{Distributed Geometric Median}
\label{sec:geommedian}

The geometric median of a discrete dataset $\cX \subset \bbR^D$ is defined as
\begin{equation}
\label{eq:g_median}
\argmin \limits_{\vy \in \bbR^D} \sum \limits_{\vx \in \cX} \Vert \vx - \vy \Vert.
\end{equation}
Weiszfeld's algorithm~\cite{weiszfeld1937sur} is a common numerical approach to approximating  \eqref{eq:g_median} within a sufficiently small error. It applies an iteratively reweighted least squares (IRLS) procedure. However,  if in one of the iterations, the estimate coincides with one of the data points,  then Weiszfeld's algorithm fails to converge to the geometric median. To avoid this issue, we consider the following regularized version of \eqref{eq:g_median}:
\begin{equation}
\label{eq:reg_g_median}
\argmin \limits_{\vy \in \bbR^D} \sum \limits_{\vx \in \cX, \Vert \vx - \vy \Vert \ge \delta}  \Vert \vx - \vy \Vert + \sum \limits_{\vx \in \cX, \Vert \vx - \vy \Vert < \delta} \left( \frac{\Vert \vx - \vy \Vert^2}{2 \delta} + \frac{\delta}{2} \right),
\end{equation}
where $\delta > 0$ is a small regularization parameter. We can solve \eqref{eq:reg_g_median} by the generalized Weiszfeld's algorithm~\cite[\S 4]{chanmul}. This algorithm runs as follows: it starts with an initial guess $\vy_0 \in \bbR^D,$ and at iteration $s \ge 1$ it computes
$$\vy_s = \sum \limits_{\vx \in \cX} \frac{\vx}{\max \left( \Vert \vx - \vy_{s - 1} \Vert, \delta \right)} \biggm/ \sum \limits_{\vx \in \cX} \frac{1}{\max \left( \Vert \vx - \vy_{s - 1} \Vert, \delta \right)}.$$
The  sequence $\{\vy_s\}_{s \in \bbN}$ $r$-linearly converges to the solution of \eqref{eq:reg_g_median} (see~\cite{chanmul}).

We assume a dataset $\cX$ with $\{ \cX_k\}_{k=1}^K$ distributed at $K$ nodes, and distributedly compute  the regularized geometric median of $\cX$ by CBGA. In view of \S\ref{sec:cbga}, it is enough to compute the dual function of \eqref{eq:reg_g_median} at each node, that is, compute for each $1 \le k \le K$
\begin{equation}
\label{eq:distgeommeddual}
d_k(\blam) = \min \limits _{\vy \in \bbR^D} \sum \limits_{\substack{\vx \in \cX_k, \\ \Vert \vx - \vy \Vert \ge \delta}}  \Vert \vx - \vy \Vert + \sum \limits_{\substack{\vx \in \cX_k, \\ \Vert \vx - \vy \Vert < \delta}} \left( \frac{\Vert \vx - \vy \Vert^2}{2 \delta} + \frac{\delta}{2} \right) + \sum_{m \in \cE_k} c_{mk} \blam_m^T \vy,
\end{equation}
where $\blam = [\blam_1^T, \dots, \blam_M^T ]^T \in \bbR^{MD}.$ We suggest solving  \eqref{eq:distgeommeddual} by IRLS as follows: start with an initial guess $\vy_k^0 \in \bbR^D$ and at iteration $s \ge 1$ compute
$$\vy_k^s = \left( 2 \sum \limits_{\vx \in \cX_k} \frac{\vx}{\max \left( \Vert \vx - \vy_k^{s - 1} \Vert, \delta \right)} - \sum_{m \in \cE_k} c_{mk} \blam_m \right) \biggm/ \left( 2 \sum \limits_{\vx \in \cX_k} \frac{1}{\left( \Vert \vx - \vy_k^{s - 1} \Vert, \delta \right)} \right).$$
The convergence of $\{\vy_k^s\}_{s \in \bbN}$ follows from that of IRLS (see \cite{chanmul}) and CBGA (see \S \ref{sec:cbga}).

\subsection{CADMM Solution for the Distributed GMS Problem}
\label{sec:cadmm}

We formulate in \Cref{algo:cadmmgms} a CADMM solution of the distributed GMS problem by following the CADMM scheme of \cite{mateos2010distributed}.
The solution of the local problem is discussed in \S\ref{sec:sol_eq_admm}.
%of this algorithm %and its guarantees, are
%is the same as the one for the local problem of \Cref{algo:consgms}.

\begin{algorithm}[H]
\caption{CADMM implementation for distributed GMS (CADMM-GMS)}
\label{algo:cadmmgms}

\begin{algorithmic}
 \STATE \textbf{Input:} Network with $K$ nodes, $\cX_1, \dots, \cX_K:$ datasets in the $K$ nodes, $T_{CADMM}$, $T_{GMS}$: stopping iteration numbers, $\delta$: regularization parameter (default: $10^{-10}$) and $\rho$: penalty parameter for CADMM

\STATE \textbf{Set:}   For all $1 \le k \le K$, $\mZ_k^0 = \v0$ and $\mQ_k^0$ is the solution of \Cref{algo:solveequation1} with input $\cX_k,$ $\mA_k = \v0,$ $T_{GMS}$ and $\delta$

\FOR{$s = 1:T_{CADMM}$}
	\STATE
\begin{itemize}
	\item For $1 \le k \le K$ update $\mZ_k^s$ by
	\begin{equation}
\label{eq:zks}
		\mZ_k^s = \mZ_{k}^{s-1} + \rho \sum_{j \in \cN_k} \left( \mQ_k^{s-1} - \mQ_j^{s-1} \right)
	\end{equation}	
		
	\item For $1 \le k \le K$ apply the algorithm described in \S\ref{sec:sol_eq_admm} to solve
\begin{equation}
	\label{eq:admmlocalobj}
	\mQ_k^{s} = \argmin \limits_{\mQ_k \in \bbH} \tilde{G}_{ADMM}(\mQ_k), \text{ where }
\end{equation}
$\tilde{G}_{ADMM}(\mQ_k) = F_k(\mQ_k) + \tr \left( \mQ_k^T \mZ_k^s \right) + \rho \sum\limits_{j \in \cN_k} \left\Vert \mQ_k - \frac{\mQ_k^{(s-1)} + \mQ_j^{(s-1)}}{2}  \right\Vert_2^2$
\end{itemize}
\ENDFOR
\RETURN $L_k := $ the span of the bottom $d$ eigenvectors of $\mQ_k^{T_{CADMM}}, 1 \le k \le K$
\end{algorithmic}
\end{algorithm}

\subsubsection{Algorithm for computing the solution of \eqref{eq:admmlocalobj}}
\label{sec:sol_eq_admm}

We propose an iterative scheme for solving \eqref{eq:admmlocalobj}, which is almost identical to \Cref{algo:solveequation1}, but at each iteration $s$ instead of finding the trace one solution of \eqref{eq:lyappsd}, it finds the trace one solution of the following Lyapunov equation in $\mP$:
\begin{multline}
\mP \left( \sum \limits_{\vx \in \cX_k} \frac{\vx \vx^T}{2 \max(\Vert\mQ \vx\Vert, \delta)} + \rho |\cN_k| \mI \right) + \left( \sum \limits_{\vx \in \cX_k} \frac{\vx \vx^T}{2 \max(\Vert\mQ \vx\Vert, \delta)} + \rho |\cN_k| \mI \right) \mP \\ + \mZ_k^s - \rho \sum_{j \in \cN_k} \left( \mQ_k^{s-1} + \mQ_j^{s-1} \right) = c \mI.
\label{eq:cadmmlyapeq}
\end{multline}
Here, $c$ is chosen so that $\tr(\mP)=1$ and its existence is guaranteed by Lemma~\ref{lem:lemma1}. The convergence theory for this iterative algorithm is the same as the one developed for \Cref{algo:solveequation1}.

\section{Supplementary Details}
\label{append:proofs}
\subsection{On the Minimizer of \eqref{eq:dpca2}}
\label{append:lem1proof_whole}
We first state the main result of this section:
%e following lemma guarantees the unique minimizer of \eqref{eq:dpca2}.
\begin{lemma}
\label{lem:pcalem1}
If $\cX_k \subset \bbR^D$ is full rank and $\mA_k \in \cS$, then the minimizer of \eqref{eq:dpca2} is unique. Furthermore, there exists
a unique $c' \in \bbR$ such that this minimizer is the unique solution of the following equation with $c=c'$
\begin{equation}
\mQ \left( \sum \limits_{\vx \in \cX_k} \vx \vx^T \right) + \left( \sum \limits_{\vx \in \cX_k} \vx \vx^T \right) \mQ + \mA_k = c \mI.
\label{eq:dpcalyap}
\end{equation}
\end{lemma}
Section \ref{append:prlemmaslem3} states and proves a lemma about the solution of the above Lyapunov equation and \S\ref{append:lem1proof} then uses this latter lemma to conclude \cref{lem:pcalem1}. At last, \S\ref{append:find_min} briefly discusses the computation of the minimizer of \eqref{eq:dpca2}.

\subsubsection{Preliminary lemma}
\label{append:prlemmaslem3}
We verify the following lemma.
% on the solution of Lyapunov equation, whose first part is well-known (see page 107 of \cite{lyapbhat}).
\begin{lemma}
If $c \in \bbR,$ $\mX \in \cSpp$ and $\mA \in \cS$, then the following Lyapunov equation
\begin{equation}
\mQ \mX + \mX \mQ + \mA = c \mI
\label{eq:lyapinc}
\end{equation}
has a unique solution in $\mQ \in \cS$. Furthermore, $\tr(\mQ)$ is an increasing linear function of $c$ with slope $\tr(\mX^{-1})/2.$
\label{lem:lemma1}
\end{lemma}
\begin{proof}
The existence and uniqueness of the solution of \eqref{eq:lyapinc} is well-known~\cite[page 107]{lyapbhat}. We thus only need to show that $\tr(\mQ)$ is an increasing linear function of $c$. Assume that $\mQ_1$ and $\mQ_2$ are the solutions of \eqref{eq:lyapinc} corresponding to $c_1$ and $c_2,$ that is,
\begin{equation}
\mQ_1 \mX + \mX \mQ_1 + \mA = c_1 \mI \text{ and } \mQ_2 \mX + \mX \mQ_2 + \mA = c_2 \mI.
\label{eq:lyaplem}
\end{equation}
%We need to show that $(\tr(\mQ_1) - \tr(\mQ_2))/(c_1 - c_2)$ is a positive constant that does not depend on $c_1, c_2, \mQ_1, \mQ_2$.
 %therefore $\tr(\mQ)$ is linear in $c.$
Subtracting the two equations in \eqref{eq:lyaplem}, results in
\begin{equation}
(\mQ_1 - \mQ_2) \mX + \mX (\mQ_1 - \mQ_2) = (c_1 - c_2) \mI,
\label{eq:lyapincdif}
\end{equation}
whose unique solution is $(\mQ_1 - \mQ_2) = (c_1 - c_2) \mX^{-1}/2$. % is the unique solution of \eqref{eq:lyapincdif}.
% (with coefficient $\mX \in \cSpp$) and \eqref{eq:lyapincdif} is a Lyapunov equation, it is the unique solution in $\cS.$
By taking traces of both sides of the solution, we get that $(\tr(\mQ_1) - \tr(\mQ_2))/(c_1 - c_2) = \tr(\mX^{-1})/2>0$.
 %Since $\mX \in \cSpp$ it follows that $\mX^{-1} \in \cSpp$ and therefore $\tr(\mX^{-1}) > 0.$
\end{proof}

\subsubsection{Proof of \Cref{lem:pcalem1}}
\label{append:lem1proof}
Since $\cX_k$ is full rank, $\sum_{\vx \in \cX_k} \vx \vx^T \in \cSpp.$
Hence the minimized function in  \eqref{eq:dpca2} is strongly convex and its minimizer is unique.

We note that \eqref{eq:dpcalyap} is a Lyapunov equation in $\mQ.$ \Cref{lem:lemma1} implies that there is a unique value $c'$ for which the unique solution of \eqref{eq:dpcalyap} has trace $1.$ We denote this solution by $\mQ'.$ Next, we show that $\mQ'$ is the minimizer of \eqref{eq:dpca2}. The following two facts: $\sum_{\vx \in \cX_k} \Vert\mQ \vx_k\Vert^2 + \tr(\mA_k \mQ) = \sum_{\vx \in \cX_k} \tr (\mQ \vx_k \vx_k^T \mQ) + \tr(\mA_k \mQ)$ for $\mQ \in \bbH$ and $\tr(\mQ) = 1$ for $\mQ \in \bbH,$ imply the same minimizer for \eqref{eq:dpca2} and
\begin{equation}
\label{eq:changedobjective}
\min_{\mQ \in \bbH} l(\mQ), \text{ where } l(\mQ) = \sum_{\vx \in \cX_k} \tr (\mQ \vx_k \vx_k^T \mQ) + \tr(\mA_k \mQ) - c' \tr(\mQ).
\end{equation}
Since $l(\mQ)$ is convex on $\bbH$, we conclude that $\mQ'$ minimizes \eqref{eq:changedobjective} by showing that the derivative of $l(\mQ)$ at $\mQ'$, when restricted to $\bbH$,  is $\v0$:
\begin{equation}
\frac{d}{d \mQ} l(\mQ) \biggr\vert_{\mQ = \mQ'} = \mQ' \left( \sum_{\vx \in \cX_k} \vx \vx^T \right) + \left( \sum_{\vx \in \cX_k} \vx \vx^T \right) \mQ' + \mA_k - c' \mI = \v0. \tag*{\QED}
\end{equation}
%This implies that all directional derivatives of $l(\mQ)$ restricted to $\bbH$ are $\v0.$ Thus, $\mQ'$ solves \eqref{eq:changedobjective} and \eqref{eq:dpca2}.
%\qed

\subsubsection{Computing the Minimizer of \eqref{eq:dpca2}}
\label{append:find_min}
In view of \Cref{lem:lemma1} we compute $c'$ and the corresponding solution of \eqref{eq:lyaplem} as follows.
We solve \eqref{eq:dpcalyap} with $c=0$ to obtain $\mQ_* \in \cS$. We  then use $\tr(\mQ_*)$ and the slope $\tr(\mX^{-1})/2$, where
$\mX = \sum_{\vx \in \cX_k} \vx \vx^T$, to find $c'$. % for which the solution of  \eqref{eq:lyaplem} is the minimizer of \eqref{eq:dpca2}.
Therefore, computing this minimizer requires computing $\mX$, which costs $O (N_{\max} \times D^2)$, and solving two Lyapunov equations, which costs $O(D^3)$ (see~\cite{bartels1972solution}).

\subsection{Proof of \cref{lem:lemmapsd}}

\label{append:lem2}

%Equation \eqref{eq:lyappsd} is Lyapunov.
Let $\mX = \sum _{i = 1}^N \vx_i \vx_i^T / \left( 2 \max(\Vert\mQ \vx_i\Vert, \delta) \right)$ and %$\mW = \mA - c\mI$. We
note that $\mX \in \cSpp$.
 This observation and \Cref{lem:lemma1} imply that there is a unique value $c \in \bbR$ for which \eqref{eq:lyappsd} has a unique solution in $\bbH$.
We will show that $c > \lambda_1(\mA)$, equivalently $\mA - c\mI \preceq \boldsymbol{0}$, and thus in view of \cite[page 107]{lyapbhat}, this solution is in $\cSpp$.

To get this estimate, we rewrite \eqref{eq:lyappsd} as $\mP + \mX \mP \mX^{-1} + \mA \mX^{-1} = c \mX^{-1}.$ Applying trace to both sides and using the following facts: $\tr(\mP) = 1,$ $\tr(\mX \mP \mX^{-1}) = \tr(\mX^{-1} \mX \mP) = 1$ and $\tr(\mA \mX^{-1}) \ge \lambda_D(\mA) \tr(\mX^{-1})$ yields $c \ge 2 / \tr(\mX^{-1}) + \lambda_D(\mA).$

Let $\mX_{*} = \sum _{i = 1}^N \vx_i \vx_i^T / \left( 2 \max(\Vert\vx_i\Vert, \delta) \right).$ Since $\mQ \in \cSp \cap \bbH$ and $\max(\Vert\mQ \vx_i\Vert, \delta) \le \max(\Vert\vx_i\Vert, \delta)$ for $1 \le i \le N,$ $\mX - \mX_{*} \in \cSp,$ which implies that $\mX_{*}^{-1} - \mX^{-1}  \in \cSp.$ Combining the last result with \eqref{eq:gmsthcond}, $\Vert \mA \Vert_2 > \lambda_1(-A)$ and the estimate of $c$ we obtain that $c \ge 2 / \tr(\mX_{*}^{-1}) + \lambda_D(\mA) \ge 2 \Vert \mA \Vert_2 - \lambda_1(-\mA) \ge \lambda_1(\mA).$

The last statement of the lemma is a direct application of \Cref{lem:lemma1}.
\QED

%----------------------------------------------------------------------------------------

\subsection{On the Choice of the Step-Size}
\label{append:lemscalecond}

In view of \cref{lem:lemmapsd}, we require that condition \eqref{eq:gmsthcond} holds at each iteration of \Cref{algo:consgms} and each node $k$.
The following lemma shows that a choice of a sufficiently small step-size $\mu$ guarantees this requirement. After verifying this lemma, we discuss weaker restrictions on the step-size as well as a weaker practical version of condition \eqref{eq:gmsthcond}.

\begin{lemma}
\label{lem:scaledata}
If $\{\cX_k \}_{k=1}^K \subset \bbR^D$ are datasets distributed at $K$ nodes, $n \in \bbN$ and
\begin{equation}
\mu \le \frac{1}{n \cdot \max \limits_{1 \le k \le K}|\cE_k| \cdot \tr \left( \left(\sum \limits _{x \in \cX_k} \frac{\vx \vx^T}{\max(\Vert\vx\Vert, \delta)} \right)^{-1} \right)},
\label{eq:alpcond}
\end{equation}
then at each iteration $s \le n$ of \Cref{algo:consgms} and node $k$, $\mA_k^s$ satisfies condition \eqref{eq:gmsthcond}.
\end{lemma}
\begin{proof}
%We note that multiplying $\cX_k$ by a constant increases the RHS of \eqref{eq:gmsthcond} by this constant. The matrix $\mA_k^{s}$ %,  which is defined in \eqref{eq:dpca2}, is independent of such scaling. Therefore, we can fulfill condition \eqref{eq:gmsthcond} by initially scaling the data at the $k$th node with a sufficiently large constant. We show it is enough to scale the data at the $k$th node by $\alpha_k$.
We estimate the LHS of \eqref{eq:gmsthcond} at iteration $s$ as follows:
\begin{equation}
\Vert\mA_k^{s}\Vert_2 = \left\Vert\sum \limits_{m \in \cE_k} c_{mk} \blamb_m^{s}\right\Vert_2 \le \sum \limits_{m \in \cE_k} \Vert\blamb_m^{s}\Vert_2.
\label{eq:Abound1}
\end{equation}
In order to evaluate $\Vert\blamb_m^{s}\Vert_2$ for $1 \le m \le M,$ we apply \eqref{eq:lambdupdate} and basic inequalities:
\begin{multline}
\Vert\blamb_m^{s}\Vert_2 = \Vert\blamb_m^{s-1} + \mu (c_{mk} \mQ_k^{s} - c_{mk} \mQ_q^{s})\Vert_2 \le \Vert\blamb_m^{s-1}\Vert_2 +  \mu \Vert\mQ_k^{s} - \mQ_q^{s}  \Vert_2 \le \\
\Vert\blamb_m^{s-1}\Vert_2 +  \mu \max\{\Vert\mQ_k^{s}\Vert_2, \Vert\mQ_q^{s}\Vert_2 \} \le \Vert\blamb_m^{s-1}\Vert_2 + \mu \le  \dots \le s \mu \le n \mu.
\label{eq:Abound2}
\end{multline}
Combining \eqref{eq:alpcond}, \eqref{eq:Abound1} and \eqref{eq:Abound2} results in %the bound
\begin{equation}
\Vert\mA_k^{s}\Vert_2 \le \sum \limits_{m \in \cE_k} \Vert\blamb_m^s\Vert_2 \le |\cE_k| n \mu  \le 1 \biggr/ \tr \left( \left(\sum \limits _{\vx \in \cX} \frac{\vx \vx^T}{2\max(\Vert\vx\Vert, \delta)} \right)^{-1}\right).
\end{equation}
\end{proof}

In practice one may apply several iterations with the same fixed step-size and gradually reduce it until it satisfies the estimate above.
Nevertheless, this estimate represents a worse-case scenario and typically we expect an improved one.
Indeed, first note that condition \eqref{eq:gmsthcond} represents a worse-case scenario. In the proof of \cref{lem:lemmapsd} we used the worst-case estimate $\Vert \mQ \Vert \leq 1$. However, typically $\Vert \mQ \Vert \sim 1/D$. This will introduce a multiplicative factor $D$ for the RHS of \eqref{eq:gmsthcond} and thus of \eqref{eq:alpcond}.
Second, in \eqref{eq:Abound2} we used the estimate $\Vert \mQ^s_k - \mQ^s_q \Vert \le \max\{\Vert\mQ_k^{s}\Vert_2, \Vert\mQ_q^{s}\Vert_2 \} \le 1.$ However, typically for $\mQ^s_k, \mQ^s_q \in \bbH \cap \cSpp,$ $\max\{\Vert\mQ_k^{s}\Vert_2, \Vert\mQ_q^{s}\Vert_2 \} \sim 1/D.$ This observation introduces another multiplicative factor $D$ for the RHS of \eqref{eq:alpcond}. These two observations suggest, in practice, the following choice of a step-size:
\begin{equation}
\label{eq:step_size_practice}
\mu = \frac{D^2}{n \cdot \max \limits_{1 \le k \le K}|\cE_k| \cdot \tr \left( \left(\sum \limits _{x \in \cX_k} \frac{\vx \vx^T}{\max(\Vert\vx\Vert, \delta)} \right)^{-1} \right)}.
\end{equation}
Third of all, we note that for sufficiently small step-sizes the gradient descent gets closer to the solution, that is, $\Vert\mQ_k^{s} - \mQ_q^{s}  \Vert_2 \to 0,$ for $1 \le k, q \le K$. However, we used $1/D$ as an upper bound for $\Vert\mQ_k^{s} - \mQ_q^{s}  \Vert_2.$
At last, we comment that while the above analysis aims to guarantee that at each iteration the solution is in $\bbH \cap \cSpp$ (since \eqref{eq:gmsthcond} guarantees this), in practice it is not a main concern for small step-sizes and large number of iterations. Indeed, the solution of \eqref{eq:maxprob} coincides with the solution of GMS for the total data, which is in $\bbH \cap \cSpp$. Thus, by choosing the step-size small enough we will always converge to the solution.

%
%In practice, our code uses the step-size advocated in \eqref{eq:step_size_practice}, even though our theory requires a step-size smaller by a factor of $D^2$ to fully guarantee the existence of the solution of the local objective function. We are unable to guarantee convergence of the CBGA algorithm with any fixed choice of a constant, but our numerical results indicated nice convergence with the suggested step-size in \eqref{eq:step_size_practice}.
%
%----------------------------------------------------------------------------------------

\subsection{Proof of Theorem \ref{th:itgmsthreg}}
\label{sec:thproof}

We establish an auxiliary lemma in \S\ref{sec:preliminaries} and conclude Theorem \ref{th:itgmsthreg} in \S\ref{sec:th1proof} by following ideas of \cite{chanmul,gms} and using this lemma.

\subsubsection{Preliminary Proposition}
\label{sec:preliminaries}
We first apply \Cref{lem:lemma1} to define the mapping $T_{\mA}(\mQ)$ and then establish the continuity of $T_{\mA}(\mQ)$ in $\cSpp.$

\begin{mydef}[The mapping $T_{\mA}(\mQ)$]
If $\{\vx_i\}_{i=1}^N \subset \bbR^{D},$ $\delta > 0,$ $\mQ \in \cSp \cap \bbH$ and $\mA \in \cS$ with $\tr(\mA) = 0,$ then $T_{\mA}(\mQ)$ is the solution of the following equation in $\mP$
\begin{equation}
\mP \left( \sum \limits_{i = 1}^N \frac{\vx_i \vx_i^T}{\max(\Vert\mQ \vx_i\Vert, \delta)} \right) + \left( \sum \limits_{i = 1}^N \frac{\vx_i \vx_i^T}{\max(\Vert\mQ \vx_i\Vert, \delta)} \right) \mP + \mA = c \mI,
\label{eq:reflyapconv}
\end{equation}
where $c = c(\mQ) \in \bbR$ is uniquely chosen so that the solution has trace 1.
\label{def:itstep}
\end{mydef}
%Note that \Cref{lem:lemma1} guarantees the existence of unique $c(Q)$ in Definition \ref{def:itstep}.
\begin{lemma}
Assume a sequence $\{\mQ^t\}_{t \in \bbN} \subset \cSpp \cap \bbH$, $\mA \in \cS$ with $\tr(\mA) = 0,$ $\{\vx_i\}_{i=1}^N \subset \bbR^{D}$ and $\delta > 0.$ If $\mQ^t \to \hat{\mQ},$ then $T_{\mA}(\mQ^t) \to T_{\mA}(\hat{\mQ}).$
\label{th:itconvlemma}
\end{lemma}

\begin{proof}

For $t \in \bbN,$ let $\mP^t = T_{\mA}(\mQ^t)$ be the trace one solution of \eqref{eq:reflyapconv} with $\mQ = \mQ^t$ and $c = c^t.$  Let $\hat{\mP} = T_{\mA}(\hat{\mQ})$ be the trace one solution of \eqref{eq:reflyapconv} with $\mQ = \hat{\mQ}$ and $c = \hat{c}$.
We need to prove that $\mP^t \to \hat{\mP}$ as $t \to \infty$. We write \eqref{eq:reflyapconv} with $\mP^t, \mQ^t$ and $c^t$  as
\begin{equation}
\label{eq:reflyapconv1}
\mP^t \left( \sum \limits_{i = 1}^N \frac{\vx_i \vx_i^T}{\max(\Vert\mQ^t \vx_i\Vert, \delta)} \right) + \left( \sum \limits_{i = 1}^N \frac{\vx_i \vx_i^T}{\max(\Vert\mQ^t \vx_i\Vert, \delta)} \right) \mP^t + \mA = c^t \mI.
\end{equation}
Note that $\mR^t := \sum_{i = 1}^N \vx_i \vx_i^T / \max(\Vert\mQ^t \vx_i\Vert, \delta) \to \hat{\mR} := \sum_{i = 1}^N \vx_i \vx_i^T / \max(\Vert\mQ \vx_i\Vert, \delta)$
as $t \to \infty$. Also observe that for $\mQ = \hat{\mQ},$ $c = \hat{c}$ and $T_{\mA}(\mQ^t) = \mP^t$, \eqref{eq:reflyapconv}  has the form
\begin{equation}
\mP^t \mR^t + \mR^t \mP^t + \mA = c^t \mI.
\label{eq:denlyapeq1}
\end{equation}
By subtracting $c^t \mI$ from both sides of \eqref{eq:denlyapeq1} and rewriting
$c^t \mI = c^t {\mR^t}^{-1} \mR^t/2 + \mR^t c^t {\mR^t}^{-1}/2$,
\eqref{eq:denlyapeq1} becomes
%\begin{equation}
%\label{eq:reflyapconv2}
$(\mP^t - c^t {\mR^t}^{-1} / 2) \mR^t + \mR^t (\mP^t - c^t {\mR^t}^{-1} / 2) + \mA = \v0$.
%\end{equation}
Similarly,
%\begin{equation}
%\label{eq:reflyapconv3}
$(\hat{\mP} - \hat{c} {\hat{\mR}}^{-1} / 2 ) \hat{\mR} + \hat{\mR}(\hat{\mP} - \hat{c} {\hat{\mR}}^{-1} / 2) + \mA = \v0$.
%\end{equation}
Since $\mA$ is fixed and $\mR^t \to \hat{\mR}$ as $t \to \infty,$  it follows from the last two expressions that
\begin{equation}
\mP^t - c^t {\mR^t}^{-1} / 2 \to \hat{\mP} - \hat{c} {\hat{\mR}}^{-1} / 2 \text{  as  } t \to \infty.
\label{eq:reflyapconv4}
\end{equation}
By taking the trace of both sides of \eqref{eq:reflyapconv4} and using the facts that $\tr(\mP^t) = \tr(\hat{\mP}) = 1$ and $\hat{\mR}^t \to \mR$ as $t \to \infty$, we get that $c^t \to \hat{c}$ and consequently $\mP^t \to \hat{\mP}$ as $t \to \infty.$
\end{proof}

\subsubsection{Conclusion of Theorem \ref{th:itgmsthreg}}
\label{sec:th1proof}

We divide the proof of Theorem \ref{th:itgmsthreg} into the following steps suggested in \cite{gms}.

\textit{Step 1:} \textbf{The majorizing function H and its minimizer.} Let $H_k^{\delta}$ denote the following function
\begin{equation}
\label{eq:defofh}
H_{k}^{\delta} (\mQ, \mQ^*) = \sum \limits_{\vx \in \cX_k} \left( \frac{\Vert \mQ \vx \Vert^2}{2 \max(\Vert \mQ^* \vx \Vert, \delta)} + \frac{\max(\Vert \mQ^* \vx \Vert, \delta)}{2} \right) + \tr(\mQ \mA_k).
\end{equation}
We show next that $H_k^{\delta}$ majorizes $G_k^{\delta},$ that is,
\begin{equation}
H_k^{\delta} (\mQ, \mQ) = G_k^{\delta} (\mQ) \text{  and  } G_k^{\delta}(\mQ) \le H_k^{\delta}(\mQ, \mQ^{*}).
\label{eq:GHeq}
\end{equation}
The above equality is immediate. To prove the above inequality %in \eqref{eq:GHeq}
we define
\begin{gather*}
G_k^{\delta}(\vx, \mQ) =
\begin{cases}
\Vert\mQ \vx\Vert, & \text{if } \Vert \mQ \vx\Vert \ge \delta;\\
\frac{\Vert\mQ \vx\Vert^2}{2 \delta} + \frac{\delta}{2}, & \text{if }  \Vert \mQ \vx \Vert < \delta,\\
\end{cases} \\
H_k^{\delta}(\vx, \mQ, \mQ^*) = \frac{\Vert \mQ \vx \Vert^2}{2 \max(\Vert\mQ^* \vx\Vert, \delta)}  + \frac{\max(\Vert\mQ^* \vx\Vert, \delta)}{2}.
\end{gather*}
We show that $G_k^{\delta} (\vx, \mQ) \le H_k^{\delta}(\vx, \mQ, \mQ^{*})$ by considering four complementing cases:
\begin{enumerate}

\item[Case 1:] $\Vert\mQ \vx\Vert \ge \delta$ and $\Vert\mQ^* \vx\Vert \ge \delta.$ In this case
$$
G_k^{\delta}(\vx, \mQ) = \Vert\mQ \vx\Vert = \frac{\Vert\mQ \vx\Vert \Vert\mQ^* \vx\Vert}{\Vert\mQ^* \vx\Vert} \le \frac{\Vert\mQ \vx\Vert^2 + \Vert\mQ^* \vx\Vert^2}{2\Vert\mQ^* \vx\Vert} = H_k^{\delta}(\vx, \mQ, \mQ^*).
$$
\item[Case 2:] $\Vert\mQ \vx\Vert \ge \delta$ and $\Vert\mQ^* \vx\Vert < \delta.$ We conclude the desired property as follows $0 \le \left( \Vert\mQ\vx\Vert - \delta \right)^2 = \Vert\mQ\vx\Vert^2 - 2\Vert\mQ\vx\Vert \delta + \delta^2 = \delta \left(H_k^{\delta}(x, \mQ, \mQ^*) - G_k^{\delta}(\vx, \mQ)\right).$
\item[Case 3:] $\Vert\mQ \vx\Vert < \delta$ and $\Vert\mQ^* \vx\Vert \ge \delta.$ In this case
\begin{multline*}
G_k^{\delta}(\vx, \mQ) - H_k^{\delta}(\vx, \mQ, \mQ^*) = \frac{1}{2} \left( \frac{\Vert\mQ\vx\Vert^2}{\delta} + \delta - \frac{\Vert\mQ \vx\Vert^2}{\Vert\mQ^*x\Vert} - \Vert\mQ^*x\Vert \right) = \\ \frac{\Vert\mQ^* \vx\Vert - \delta}{2} \left( \frac{\Vert\mQ\vx\Vert^2}{\delta \Vert\mQ^*\vx\Vert} - 1 \right) \le 0.
\end{multline*}
\item[Case 4:] $\Vert\mQ \vx\Vert < \delta$ and $\Vert\mQ^* \vx\Vert < \delta.$ Then
$G_k^{\delta}(\vx, \mQ) = H_k^{\delta}(\vx, \mQ, \mQ^*).$
\end{enumerate}
We thus conclude \eqref{eq:GHeq} as follows
\begin{equation}
G_k^{\delta}(\mQ) = \sum \limits_{\vx \in \cX_k} G_k^{\delta}(\vx, \mQ) + \tr(\mQ \mA_k) \le H_k^{\delta}(\vx, \mQ, \mQ^*) + \tr(\mQ \mA_k) = H_{k}^{\delta} (\mQ, \mQ^*).
\label{eq:regineqmaj2}
\end{equation}

Next, we claim that the minimizer of $H_k^{\delta}(\mQ,\mQ_k^t)$ over all $\mQ \in \bbH$ is $\mQ_k^{t+1}.$ First we note that since the data satisfies the two-subspaces criterion and since $\tr(\mA_k \mQ_k)$ is a linear function, then according to Theorem 2 of \cite{gms}, $H_k^{\delta}(\mQ, \mQ_k^t)$ is strictly convex over $\mQ \in \bbH.$ We further note that for $\mQ \in \bbH$, $H_{k}^{\delta} (\mQ, \mQ^*) = \tilde{H}_{k}^{\delta} (\mQ, \mQ^*),$ where
\begin{equation}
\label{eq:defofhtild}
\tilde{H}_{k}^{\delta} (\mQ, \mQ^*) =  \sum \limits_{\vx \in \cX_k} \left(\frac{\tr (\mQ \vx \vx^T \mQ)}{2\max(\Vert\mQ^* \vx\Vert, \delta)} + \frac{\max(\Vert\mQ^* \vx\Vert, \delta)}{2}\right) + \tr(\mQ \mA_k).
\end{equation}
Therefore, the minimizers over $\bbH$ of $H_k^{\delta}(\mQ, \mQ_k^t)$ and $\tilde{H}_k^{\delta}(\mQ, \mQ_k^t) - c_k \tr(\mQ)$ are the same.
We compute the derivative of the latter term w.r.t.~$\mQ$ as follows:
\begin{multline}
\label{eq:hderivreg}
\frac{d}{d\mQ} \left( \tilde{H}_k^{\delta}(\mQ, \mQ_k^t) - c_k \tr(\mQ) \right) \biggr\vert_{\mQ = \mQ_k^{t+1}} = \\ \frac{1}{2} \left( \mQ_k^{t+1} \sum \limits_{\vx \in \cX_k} \frac{\vx \vx^T }{\max(\Vert\mQ_k^t \vx\Vert, \delta)} + \sum \limits_{\vx \in \cX_k} \frac{\vx \vx^T }{\max(\Vert\mQ_k^t \vx\Vert, \delta)} \mQ_k^{t+1} \right) + \mA_k - c_k \mI = \v0.
\end{multline}
The last equation follows from the definition of $\mQ_k^{t+1}$ (see \eqref{eq:itlyapsolreg}). Combining this with the fact that $H^{\delta}(\mQ, \mQ_k^t)$ is strictly convex when restricted to $\mQ \in \bbH,$ we conclude that $\mQ_k^{t+1}$is the unique minimizer of $H^{\delta}(\mQ, \mQ_k^t)$ for $\mQ \in \bbH.$

\textit{Step 2:} \textbf{Convergence of $\{G_k^{\delta}(\mQ_k^t)\}_{t \in \bbN}.$} We first note that $G_k^{\delta}$ is bounded from below on $\bbH.$ Indeed, $G_k^{\delta}(\mQ) \geq
\tr(\mQ \mA_k) \ge \tr(\mQ) \times \min \eig(\mA_k) = \min \eig(\mA_k).$

Next, we show that $G_k^{\delta}(\mQ_k^t)$ decreases with $t$. By using \eqref{eq:GHeq} and the fact that $\mQ_k^{t+1}$ is the minimizer of $H_k^{\delta}(\mQ, \mQ_k^{t})$ for $\mQ \in \bbH,$ we get that
\begin{equation}
G_k^{\delta}(\mQ_k^{t+1}) \le H_k^{\delta}(\mQ_k^{t+1}, \mQ_k^{t}) \le H_k^{\delta}(\mQ_k^{t}, \mQ_k^{t}) = G_k^{\delta}(\mQ_k^{t}).
\label{eq:mainineq}
\end{equation}
Since $\{G_k^{\delta}(\mQ_k^{t})\}_{t \in \bbN}$ is bounded from below and decreases, it converges.

\textit{Step 3:} \textbf{$\Vert\mQ_k^t - \mQ_k^{t+1}\Vert \to 0$ as $t \to \infty.$} It follows from \eqref{eq:hderivreg} and the fact that $\mQ_k^t - \mQ_k^{t+1} \in \cS$ has trace $0$, that
\begin{equation*}
\tr \left( ( \mQ_k^{t+1} \sum \limits_{\vx \in \cX_k} \frac{\vx \vx^T}{\max(\Vert\mQ_k^t \vx\Vert, \delta)} + \sum \limits_{\vx \in\cX_k} \frac{\vx \vx^T }{(\Vert\mQ_k^t \vx\Vert, \delta)} \mQ_k^{t+1} + 2 \mA_k ) (\mQ^{t}_k - \mQ^{t+1}_k)\right)  = 0.
\end{equation*}
Simplifying the above equation, we get that
\begin{multline}
\tr\left(\mA_k \left( \mQ_k^t - \mQ_k^{t+1} \right) \right) = - \tr \left( \mQ_k^{t+1} \sum \limits_{\vx \in \cX_k} \frac{\vx \vx^T }{ \max(\Vert\mQ_k^t \vx\Vert, \delta)} \left( \mQ_k^{t} - \mQ_k^{t+1} \right) \right) = \\ \tr\left( \sum \limits_{\vx \in \cX_k}  \frac{\mQ_k^{t+1} \vx \vx^T ( \mQ_k^{t+1} - \mQ_k^t))}{\max(\Vert\mQ_k^t \vx\Vert, \delta)} \right) = \sum \limits_{\vx \in \cX_k}  \frac{\vx^T \mQ_k^{t+1} (\mQ_k^{t+1} - \mQ_k^t) \vx }{ \max(\Vert\mQ_k^t \vx\Vert, \delta)}.
\label{eq:traceeqsimp}
\end{multline}
It follows from \eqref{eq:mainineq} and \eqref{eq:defofh} that
\begin{multline}
\label{eq:reqequals1}
G_k^{\delta}(\mQ_k^{t}) - G_k^{\delta}(\mQ_k^{t+1}) \ge H_k^{\delta}(\mQ_k^{t}, \mQ_k^{t}) - H_k^{\delta}(\mQ_k^{t+1}, \mQ_k^{t}) = \\ \frac{1}{2} \sum \limits_{\vx \in \cX_k} \left( \frac{\Vert\mQ_k^{t} \vx\Vert^2 - \Vert\mQ_k^{t+1} \vx\Vert^2}{\max(\Vert\mQ_k^{t} \vx\Vert, \delta)} \right) + \tr((\mQ_k^t - \mQ_k^{t+1}) \mA) = \\ \frac{1}{2} \sum \limits_{\vx \in \cX_k} \left( \frac{ \vx^T (\mQ_k^{t})^2 \vx - \vx^T(\mQ_k^{t+1})^2 \vx}{\max(\Vert\mQ_k^{t} \vx\Vert, \delta)} \right) + \tr((\mQ_k^t - \mQ_k^{t+1}) \mA).
\end{multline}
The combination of \eqref{eq:traceeqsimp} and \eqref{eq:reqequals1} yields
\begin{multline}
\label{eq:convres}
G_k^{\delta}(\mQ_k^{t}) - G_k^{\delta}(\mQ_k^{t+1}) \ge  \frac{1}{2} \sum \limits_{\vx \in \cX_k} \left( \frac{ \vx^T (\mQ_k^{t})^2 \vx - \vx^T(\mQ_k^{t+1})^2 \vx}{\max(\Vert\mQ_k^{t} \vx\Vert, \delta)} \right) + \\ \sum \limits_{\vx \in \cX_k}  \frac{\vx^T \mQ_k^{t+1}   (\mQ_k^{t+1} - \mQ_k^t) \vx)}{\max(\Vert\mQ_k^t \vx\Vert, \delta)} = \frac{1}{2} \sum \limits _{\vx \in \cX_k} \frac{\Vert(\mQ_k^t - \mQ_k^{t+1}) \vx\Vert^2}{\max(\Vert\mQ^t \vx\Vert, \delta)} \ge 0.
\end{multline}
Since $\{G(\mQ_k^t)\}_{t \in \bbN}$ converges, \eqref{eq:convres} implies that
\begin{equation}
\sum \limits _{\vx \in \cX_k} \frac{\Vert(\mQ_k^t - \mQ_k^{t+1}) \vx\Vert^2}{\max(\Vert\mQ_k^t \vx\Vert, \delta)} \to 0 \text{ as } t \to \infty
\end{equation}
and consequently (using the fact that $\cSpan\{\vx\}_{\vx \in \cX_k} = \bbR^D$):
\begin{equation}
\Vert\mQ_k^{t} - \mQ_k^{t+1}\Vert \to 0 \text{ as } t \to \infty.
\label{eq:convmainres1}
\end{equation}

\textit{Step 4:} \textbf{Convergence of $\{\mQ_k^t\}_{t \in \bbN}$ to the minimizer of $G_k^{\delta}(\mQ).$} The sequence $\{\mQ_k^{t}\}_{t \in \bbN}$ lies in the compact set of positive semi-definite matrices with trace $1.$ By Bolzano-Weierstrass theorem, $\{\mQ_k^t\}_{t \ge 1}$ has a converging subsequence. Let $\tilde{\mQ}_k$ denote the limit of the subsequence. We show that
\begin{equation}
\label{eq:prst}
\tilde{\mQ}_k = \argmin \limits_{\mQ \in \bbH} G_k^{\delta}(\mQ).
\end{equation}
By \cref{th:itconvlemma} and the fact that the limits of  $G_k^{\delta}(\mQ_k^{t})$ and $G_k^{\delta}(\mQ_k^{t+1}) \equiv G_k^{\delta} (T_{\mA} (\mQ_k^t))$ are the same, we conclude that $G_k^{\delta}(\tilde{\mQ}_k) = G_k^{\delta}(T_{\mA}(\tilde{\mQ}_k)).$ Combining this result with \eqref{eq:mainineq} we get that $H_k^{\delta}(T_{\mA}(\tilde{\mQ}_k), \tilde{\mQ}_k) = H_k^{\delta}(\tilde{\mQ}_k, \tilde{\mQ}_k).$ Since $T_{\mA}(\tilde{\mQ}_k)$ is the unique minimizer of $H_k^{\delta}(\mQ, \tilde{\mQ}_k)$ among all $\mQ \in \bbH$ we get that  $T_{\mA}(\tilde{\mQ}_k) = \tilde{\mQ}_k.$ That is, $\tilde{\mQ}_k$ is the unique minimizer of $H_k^{\delta}(\mQ, \tilde{\mQ}_k)$ and $\tilde{H}_k^{\delta}(\mQ, \tilde{\mQ}_k)$ among all $\mQ \in \bbH$ and thus the directional derivatives of $\tilde{H}_k^{\delta}(\mQ, \tilde{\mQ}_k)$ with respect to $\mQ$ restricted to $\bbH$ are $\v0.$ Hence, $\tr\left(\left(\frac{d}{d \mQ}  \tilde{H}_k^{\delta}(\mQ, \tilde{\mQ}_k)|_{\mQ = \tilde{\mQ}_k} \right) \left( \mP - \tilde{\mQ}_k \right)^T\right) = 0$ and thus there exists $c \in \bbR$ such that $\frac{d}{d \mQ} \tilde{H}_k^{\delta}(\mQ, \tilde{\mQ}_k)|_{\mQ = \tilde{\mQ}_k} = c \mI$. This implies that
\begin{multline}
\label{eq:gminval1}
c \mI = \frac{1}{2} \tilde{\mQ}_k \sum \limits_{\vx \in \cX_k} \frac{\vx \vx^T }{\max(\Vert\tilde{\mQ}_k \vx\Vert, \delta)} + \frac{1}{2}\sum \limits_{\vx \in \cX_k} \frac{\vx \vx^T }{ \max(\Vert\tilde{\mQ}_k \vx\Vert, \delta)} \tilde{\mQ}_k + \mA_k = \\ \frac{d}{d \mQ} \tilde{G}_k^{\delta}(\mQ)|_{\mQ = \tilde{\mQ}_k},
\end{multline}
where
\begin{equation*}
\tilde{G}_k^{\delta}(\mQ) = \sum \limits_{\vx \in \cX_k, \Vert\mQ \vx \Vert \ge \delta} \sqrt{\tr(\mQ \vx \vx^T \mQ)} + \sum \limits_{\vx \in \cX, \Vert\mQ \vx\Vert < \delta} \left( \frac{\tr(\mQ \vx \vx^T \mQ)}{2 \delta} + \frac{\delta}{2} \right) + \tr(\mQ \mA_k).
\end{equation*}
The directional derivatives of $\frac{d}{d \mQ} \tilde{G}_k^{\delta}(\mQ)|_{\mQ = \tilde{\mQ}_k}$ restricted to $\bbH$ are
\begin{equation}
\label{eq:gtilddirder}
\tr\left(\left(\frac{d}{d \mQ}  \tilde{G}_k^{\delta}(\mQ)|_{\mQ = \tilde{\mQ}_k} \right) \left( \mP - \tilde{\mQ}_k \right)^T\right) = \tr\left(c \mI \left( \mP - \tilde{\mQ}_k \right)^T\right) = 0,
\end{equation}
where for the first equality we used \eqref{eq:gminval1} and for the last equality we used that $\mP, \tilde{\mQ}_k \in \bbH$ and thus $\tr(\mP) = \tr(\tilde{\mQ}_k) = 1.$ Equation \eqref{eq:gtilddirder} and the fact that $G_k^{\delta}(\mQ) = \tilde{G}_k^{\delta}(\mQ)$ for $\mQ \in \bbH$ imply \eqref{eq:prst}. Finally, combining \eqref{eq:convmainres1}, \eqref{eq:prst}, the definition of $\tilde{\mQ}_k$ and \cite[Theorem 2.1]{ostrowski}, we conclude that $\mQ_k^t \to \tilde{\mQ}_k$ as $t \to \infty.$

\textit{Step 5:} \textbf{$r$-linear Convergence.} The proof of $r$-linear convergence of $\mQ_k^t$ follows from Theorem 6.1 of \cite{chanmul} (similarly to the proof of Theorem 11 of \cite{gms}). To show that the conditions of the theorem are satisfied we just need to check that the functions $G$ and $H$ satisfy Hypotheses 4.1 and 4.2 of \cite{chanmul} (see proof of Theorem 6.1 in there and note that $G$ and $H$ of this work are parallel to $F$ and $H$ of \cite{chanmul}, respectively). We note that \cite{chanmul} states the result for vector-valued functions, which can be easily generalized for matrix-valued functions. Since $\mQ_k^t$ converges, it is enough to show that Hypotheses 4.1 and 4.2 hold for some local neighborhood $B(\tilde{\mQ}_k, \epsilon)$ of $\tilde{\mQ}_k,$  for some $\epsilon > 0.$ Conditions 1 and 3  of Hypothesis 4.1 are easy to check, since $G$ is twice differentiable on $B(\tilde{\mQ}_k, \epsilon)$ and $G$ is bounded from below (as we have already shown). There is no need to check condition 2, since $\mQ$ is restricted to $H.$ To verify condition 1 of Hypothesis 4.2 we need to show that
\begin{multline}
\label{eq:linconv}
H_k^{\delta}(\mQ_1,\mQ_2) = G_k^{\delta} (\mQ_2) + \tr ((\mQ_1 - \mQ_2)^{T} \frac{d}{d\mQ} G_k^{\delta} (\mQ)|_{\mQ = \mQ_2}) + \\ \frac{1}{2}\tr((\mQ_1 - \mQ_2)^{T} C(\mQ_2) (\mQ_1 - \mQ_2)).
\end{multline}
To prove \eqref{eq:linconv}, we write its RHS as follows:
\begin{flalign*}
\label{eq:hyp421}
&\sum \limits_{\vx \in \cX_k, \Vert\mQ_2 \vx\Vert\ge \delta} \Vert\mQ_2 \vx\Vert + \sum \limits_{\vx \in \cX_k, \Vert\mQ_2 \vx\Vert < \delta} \left( \frac{\Vert\mQ_2 \vx\Vert^2}{2 \delta} + \frac{\delta}{2}\right) + \tr(\mQ_2 \mA_k) + &&\nonumber \\ \nonumber
&\tr \biggr( (\mQ_1 - \mQ_2)^{T} \frac{1}{2} \biggr( \mQ_2 \sum \limits_{\vx \in \cX_k} \frac{\vx \vx^T }{\max(\Vert\mQ_2 \vx\Vert, \delta)} + \sum \limits_{\vx \in \cX_k} \frac{\vx \vx^T }{\max(\Vert\mQ_2 \vx\Vert, \delta)} \mQ_2 \biggr) + \mA_k \biggr) + &&\\
&\tr((\mQ_1 - \mQ_2)^{T} \frac{1}{2}C(\mQ_2) (\mQ_1 - \mQ_2)).&&
\end{flalign*}
By setting $C(\mQ) = \sum_{x \in X_k} \vx \vx^T / \max(\Vert\mQ \vx\Vert, \delta)$, the above equation becomes
\begin{flalign*}
&\sum \limits_{\vx \in \cX_k, \Vert\mQ_2 \vx\Vert\ge \delta} \Vert\mQ_2 \vx\Vert +  \sum \limits_{\vx \in \cX_k, \Vert\mQ_2 \vx\Vert < \delta}\left( \frac{\Vert\mQ_2 \vx\Vert^2}{2 \delta} + \frac{\delta}{2}\right) + \tr(\mQ_2 \mA_k) + &&\\\nonumber
&\tr \biggr( (\mQ_1 - \mQ_2)^{T} \mQ_2 \sum \limits_{\vx \in \cX_k} \frac{\vx \vx^T }{\max(\Vert\mQ_2 \vx\Vert, \delta)}
+ \mA_k \biggr) +  &&\\\nonumber
&\tr\left((\mQ_1 - \mQ_2)^{T} \sum \limits_{x \in X_k} \frac{\vx \vx^T} {2\max(\Vert\mQ_2 \vx\Vert, \delta)} (\mQ_1 - \mQ_2)\right) = \sum \limits_{\vx \in \cX_k, \Vert\mQ_2 \vx\Vert\ge \delta} \Vert\mQ_2 \vx\Vert + &&\\\nonumber
&\sum \limits_{\vx \in \cX_k, \Vert\mQ_2 \vx\Vert < \delta}\biggr( \frac{\Vert\mQ_2 \vx\Vert^2}{2 \delta} + \frac{\delta}{2}\biggr) + \tr(\mQ_1 \mA_k) - \sum \limits_{\vx \in \cX_k} \frac{\Vert\mQ_2 \vx\Vert^2}{\max(\Vert\mQ_2 \vx\Vert, \delta)} + &&\\\nonumber
& \sum \limits_{\vx \in \cX_k} \frac{\Vert\mQ_2 \vx\Vert^2}{2\max(\Vert\mQ_2 \vx\Vert, \delta)}  + \sum \limits_{\vx \in \cX_k} \frac{\Vert\mQ_1 \vx\Vert^2}{2\max(\Vert\mQ_2 \vx\Vert, \delta)} = \sum \limits_{\vx\in \cX_k} \frac{\Vert\mQ_1 \vx\Vert^2}{2\max(\Vert\mQ_2 x\Vert, \delta)} + &&\\\nonumber
&\sum \limits _{\vx \in \cX_k} \frac{\max(\Vert\mQ_2 \vx\Vert, \delta)}{2} +\tr(\mQ_1 \mA_k) =H(\mQ_1, \mQ_2).&&
\end{flalign*}
That is, condition 1 of Hypothesis 4.2 is verified, conditions 2 and 3 follow directly from the definition of $C(\mQ)$ and condition 4 follows from \eqref{eq:regineqmaj2}.
\QED

\bibliography{consbib}{}
\bibliographystyle{plain}

\end{document}